\documentclass{amsart}
\usepackage[svgnames]{xcolor}
\usepackage{amsfonts}
\usepackage{amsmath}
\usepackage{amsthm}
\usepackage{amssymb}
\usepackage{geometry}
\usepackage{thmtools}
\usepackage{thm-restate}
\usepackage{verbatim}
\usepackage{enumerate}
\usepackage[shortlabels]{enumitem}
\usepackage{microtype}
\usepackage{tikz-cd}
\usepackage{csquotes}
\usepackage{booktabs}
\usepackage{centernot}
\usepackage{mathtools}
\usepackage{xparse}
\usepackage{graphicx}
\usepackage{kpfonts}
\usepackage{comment}

\usetikzlibrary{knots}

\tikzset{knot/.style={double=#1,double distance=1pt,line width=2pt,white}}

\usepackage[backend=bibtex,style=alphabetic]{biblatex}
\bibliography{references.bib}

\usepackage{hyperref}
\usepackage[capitalise, noabbrev]{cleveref}

\PassOptionsToPackage{hyphens}{url}\usepackage{hyperref}

\newcounter{alltheorems}
\numberwithin{alltheorems}{section}
\newcounter{theoremtheorems}

\theoremstyle{plain}

\theoremstyle{definition}

\newtheorem{df}[alltheorems]{Definition}
\newtheorem{example}[alltheorems]{Example}
\newtheorem{construction}[alltheorems]{Construction}
\newtheorem{lemma}[alltheorems]{Lemma}

\newtheorem{corollary}[alltheorems]{Corollary}
\newtheorem{prop}[alltheorems]{Proposition}

\newtheorem{theorem}[theoremtheorems]{Theorem}

\theoremstyle{remark}

\newlength{\perspective}

\newcommand{\cube}{C}
\newcommand{\naturals}{\mathbb N}
\newcommand{\integer}{\mathbb Z}

\newcommand{\real}{\mathbb R}

\newcommand{\smashp}{\wedge}

\newcommand{\id}{\text{id}}
\newcommand{\tensor}{\otimes}

\newcommand{\disjointunion}{\bigsqcup}

\newcommand{\cfcrealization}[1]{\abs{\abs{#1}}}
\newcommand{\khovanov}{\mathcal X}
\newcommand{\tops}{\mathrm{Top}}

\newcommand{\two}[1]{\underline{2}^{#1}}

\newcommand{\twon}{\two{n}}
\newcommand{\interior}[1]{\mathring{#1}}
\newcommand{\stars}{\mathrm{Stars}}

\newcommand{\cell}{\mathrm{EX}}

\newcommand{\xto}{\xrightarrow}
\newcommand{\xot}{\xleftarrow}

\newcommand{\into}{\hookrightarrow}

\newcommand{\transform}{\rightrightarrows}

\newcommand{\ot}{\leftarrow}
\newcommand{\acts}{.}
\newcommand{\card}{\#}
\newcommand{\xinto}[1]{\xhookrightarrow{#1}}

\newcommand{\suchthat}{\mathrel{|}}

\DeclareMathOperator*{\directsum}{\bigoplus}
\DeclareMathOperator{\ext}{Ext}
\DeclareMathOperator{\kh}{Kh}
\DeclareMathOperator{\ekh}{EKh}
\DeclareMathOperator{\ckh}{CKh}
\DeclareMathOperator{\sphere}{S}

\DeclareMathOperator{\hocolim}{hocolim}

\DeclareMathOperator*{\im}{im}

\DeclareMathOperator*{\colim}{colim}
\DeclareMathOperator*{\coeq}{coeq}

\DeclareMathOperator{\setsum}{\cup}

\DeclareMathOperator*{\setintersection}{\bigcap}

\DeclareMathOperator*{\ob}{ob}
\DeclareMathOperator{\gr}{gr}

\DeclareMathOperator{\maps}{Map}
\DeclareMathOperator{\susp}{\Sigma}
\DeclareMathOperator{\modulispace}{\mathcal M}
\DeclareMathOperator{\burnside}{\mathcal B}
\DeclareMathOperator{\disjoint}{\sqcup}
\DeclareMathOperator{\permutohedron}{\Pi}
\DeclareMathOperator{\reps}{Rep}
\DeclareMathOperator{\representationring}{RO}

\DeclareMathOperator{\abelian}{Ab}
\DeclareMathOperator{\cats}{Cat}

\DeclareMathOperator{\homeo}{Homeo}
\DeclareMathOperator{\free}{\mathcal F}
\DeclareMathOperator{\conf}{Conf}
\DeclareMathOperator{\homs}{Hom}

\newcommand{\cone}[1]{\tilde{#1}}

\DeclarePairedDelimiter{\abs}{\lvert}{\rvert}
\DeclarePairedDelimiter{\corner}{\langle}{\rangle}
\newcommand{\cornered}[1]{\corner{#1}}

\NewDocumentCommand{\restrict}{m m O{}}{#1|_{#2}^{#3}}

\excludecomment{versiona}

\begin{document}

\title{Equivariant Khovanov homotopy types}
\author{Jakub Paliga}
\address{Wydział Matematyki, Informatyki i Mechaniki UW, Warszawa, ul. Banacha 2 09-027}
\email{j.paliga@uw.edu.pl}

\begin{abstract}
  We investigate external group actions on homotopy coherent diagrams. This is used to prove an equivalence between realizations of equivariant cubical flow categories and external actions on Burnside functors. In particular, the results imply that the equivariant Khovanov homotopy types defined by \cite{borodzik2021khovanov} and \cite{stoffregen2018localization} are equivariantly stably homotopy equivalent.
\end{abstract}

\maketitle


Khovanov homology was introduced in \cite{MR1740682} as a categorification of the Jones polynomial, with decategorification by way of graded Euler characteristic. Building on the work of Cohen-Jones-Segal in \cite{cjs-floers}, Lipshitz and Sarkar defined in \cite{ls2014} a space-level refinement of Khovanov homology. This takes the form of a CW spectrum \(\khovanov_{Kh}(L) = \bigvee_j \khovanov^j_{Kh}(L)\), such that for any \(j\), the cellular cochain complex of \(\khovanov^j_{Kh}(L)\) is isomorphic to the Khovanov complex \(\ckh^{\bullet, j}(L)\) in quantum grading \(j\). Among its uses, it allows for the definition of a stronger s-invariant. \cite{MR3189434}

In the case of links equipped with symmetries, it is expected that the spectra \(\khovanov(L)\) carry additional data. For periodic links, an equivariant Khovanov homotopy type was defined by \cite{borodzik2021khovanov}, who introduced equivariant cubical flow categories for the purpose. At the same time, \cite{stoffregen2018localization} proposed a different notion of equivariant Khovanov homotopy type of a periodic link, using external actions on Burnside functors. Both approaches furnish localization results relating the Khovanov homotopy type of a periodic link to the annular Khovanov homotopy type of its quotient, resulting in periodicity criteria. A difference persists in that \cite{borodzik2021khovanov} identified the Borel cohomology of their spectrum with equivariant Khovanov homology as defined by Politarczyk in \cite{politarczyk_equivariant_2019}. It has been an open question whether the equivariant spectra defined in \cite{borodzik2021khovanov} and \cite{stoffregen2018localization} are equivalent.

This paper answers the question in the affirmative. Namely, given a periodic link diagram \(D\), consider the equivariant spectra \(\khovanov_{SZ}(D)\) and \(\khovanov_{BPS}(D)\) defined by \cite{stoffregen2018localization} and \cite{borodzik2021khovanov}, respectively. We show the following.
\begin{theorem}
  \label{bigtheorem:khovanov}
  There is an equivariant stable homotopy equivalence 
  \[\khovanov_{SZ}(D) \to \khovanov_{BPS}(D).\]
\end{theorem}

The paper's structure is as follows. \Cref{sec:prereqs} introduces the cube category and the prerequisites on equivariant topology. In \cref{sec:hkoh} we recall several definitions of homotopy coherent diagrams and relate to them the concept of external action introduced in \cite{stoffregen2018localization}. \Cref{sec:flowcats} serves to describe equivariant cubical flow categories. In particular, in \cref{sec: cube flow cat} we identify the equivariant cube flow category as the free topological category on the equivariant cube \((\two{n})^m\). Burnside functors together with a notion of external action appropriate to them are introduced in \cref{sec:burnside}. 
In \cref{sec:sp-ref} we introduce configurations of stars and via a Pontrjagin-Thom-type construction associate to them equivariant maps of spheres. Those configurations are used to define geometric realizations of external actions on Burnside functors. This follows the constructions of \cite{stoffregen2018localization}, albeit allowing for more general shapes.

In \cref{sec:g-cats-external-actions} we compare the definitions of external action on Burnside functors due to \cite{musyt} and \cite{stoffregen2018localization}. In a series of comparison results, we relate those to the notion of an equivariant cubical flow category, culminating in
\begin{restatable*}{theorem}{combtheorem}
  \label{combtheorem}
  The data of an equivariant cubical flow category \((\mathcal C, f \colon \susp^V \mathcal C \to \cube_\sigma(n))\) is equivalent to that of a stable Burnside functor \((V, F \colon \two{n} \to \burnside)\) with external action.
\end{restatable*}

Passing to geometric realizations, \Cref{sec:equivalence-result} uses the results of the two previous sections to prove
\begin{restatable*}{theorem}{bigtheorem}
  \label{bigtheorem}
  Let \(\left(\mathcal C, f \colon \mathcal C \to \mathcal \cube_{\sigma}(n)\right)\) be an equivariant cubical flow category and let \(F \colon \two{n} \to \mathcal B\) be the corresponding Burnside functor with an external action. Then there is an equivariant stable homotopy equivalence \(\cfcrealization{\mathcal C} \cong \abs{F}.\)
\end{restatable*}

We finish in \cref{sec:khovanov-periodic} by describing the knot-theoretic context to which we apply the paper's results. Namely, we recall the definition of Khovanov homology by way of the Burnside functor associated to a link diagram; in the context of actions induced on a periodic diagram, \cref{bigtheorem:khovanov} is exhibited as a formal consequence of \cref{bigtheorem}.

\textbf{Acknowledgements.} This work is part of the author's thesis at University of Warsaw, supervised by Maciej Borodzik; the author is grateful for his supervisor's constant guidance. The author benefited from talks with Andrew Lobb, Dirk Schuetz, and Wojciech Politarczyk, as well as form communications with Robert Lipshitz. Part of the work was done during the author's stay at MSU, supported by IDUB Action IV.1.2. The author is grateful to Matt Stoffregen for his hospitality and many patient explanations granted during this stay. The author was also supported by OPUS grant 2019/35/B/STI/01120.


\section{Prerequisites}
\label{sec:prereqs}
\subsection{The cube category}
\label{section: the cube}
Let \(\two{} = \two{1}\) denote the poset \(\{0 > 1\}\), or the category with objects \(0\) and \(1\) and a single non-identity morphism \(1 \to 0.\) For \(n \in \integer\), \(n > 1\), we let \(\two{n} = \two{1} \times \two{n-1}.\) If \(u \geq v,\) we will denote the single element of \(\hom_{\two{n}}(u, v)\) by \(\phi_{u, v}\). For \(u \in \two{n}\), we denote by \(\abs{u}\) the \(L^1\)-norm of \(u\), so that \(\abs{u} = u_1 + \dotsb + u_n.\) If \(u \geq v\) and \(\abs{v} - \abs{u} = k\), we write \(u \geq_k v.\)
 The category \(\two{n}\) will sometimes be considered as a \(2\)-category with no non-identity \(2\)-morphisms.

A group \(G\) can be understood as a category with one object \(\ast\) and morphisms \(\hom_G(\ast, \ast) = G\) with composition defined by the group law of \(G\).
Then, an action of a group \(G\) on a small category \(\mathcal C\) is a functor \(\gamma \colon G \to \cats\) with \(\gamma(\ast) = \mathcal C\).

Although we state some results in greater generality, in application to Khovanov homology of periodic links the setting is that of a particular action of a cyclic group on a cube category.
Given integers \(n\) and \(m\), the identification \(\two{nm} \cong (\two{n})^m\) establishes a left action of \(\integer_m\) on \(\two{nm}\) by cyclic permutation of the \(\two{n}\)-factors; so that the generator \(1 \in \integer_m\) acts by
\[1 \acts (x_1, \dotsc, x_m) = (x_m, x_1, \dotsc, x_{m-1}), \quad x_1, \dotsc, x_m \in \two{n}.\]
In the setting of a group \(G\) acting on a poset \(\mathcal C\), one can speak of the fixed-point category \(\mathcal C^H\) (for any subgrop \(H \subseteq G\)). The fixed points of the action of \(\integer_m\) on \(\two{nm}\) as above, the fixed-point category \((\two{nm})^{\integer_m}\) is identified with \(\two{n}\) itself. Likewise, if \(H\subseteq \integer_m\) is the single subgroup of index \(k\), we fix an identification \((\two{nm})^H \cong \two{nk}\).


\subsection{\(G\)-cell complexes and the equivariant Spanier-Whitehead category}
The definitions are classical, and in notation we follow \cite[Sections 3.1, 3.2, 3.3]{borodzik2021khovanov}.

An orthogonal representation of a finite group \(G\) is a homomorphism \(\rho \colon G \to O(V)\), with \(V\) a real linear space equipped with an inner product. We will only consider finite-dimensional orthogonal representations and will call them ``representations'' for short. A morphism of representations \((\rho_V \colon G \to O(V)) \to (\rho_W \colon G \to O(W))\) is a linear map \(f \colon V \to W\) such that 
\(\forall g \in G \enspace \rho_W(g) \circ f = f \circ \rho_V(g).\)
Representations of \(G\) make up a monoid under the direct sum operation \(\oplus\). We will often consider virtual representations, which arise by applying the Grothendieck construction to the monoid of representations of \(G\). Namely, let
\[\representationring(G) = \{V - W \suchthat \text{$V$ and $W$ are representations of $G$}\}/{\sim},\]
where \(V_1 - W_1 \sim V_2 - W_2\) whenever \(V_1 \oplus W_2\) and \(V_2 \oplus W_1\) are isomorphic representations of \(G\). Together with multiplication induced by \(\tensor\), \(\representationring(G)\) becomes a ring.

\begin{df}
  Let \(H \subseteq G\) be a subgroup, \(V\) an \(H\)-representation. Denote by \(B_R(V)\) the ball of radius \(R\) in \(V\), which inherits the group action. A \emph{\(G\)-cell of type \((H, V)\)} is a topological space
  \[E(H, V) = G \times_H B_R(V) = G \times B_R(V) /{[(gh, x) \sim (g, hx)]}\]
  (for some \(R > 0\)), considered as a \(G\)-space via \(g' \acts [g, x] = [g'g, x]\).
  A \emph{\(G\)-cell complex} is a topological space with filtation \(X_0 \subseteq X_1 \subseteq \dotsb\) such that:
  \begin{itemize}
  \item \(X_0\) is a disjoint union of orbits \(G/{H}\),
  \item \(X_n = X_{n-1} \cup_f E(H_n, V_n)\), where \(f \colon \partial E(H_n, V_n) \to X_{n-1}\) is \(G\)-equivariant,
  \item \(X = \colim_n X_n\).
  \end{itemize}
\end{df}

Note that, should \(V\) already  be a \(G\)-representation, the cell \(E(H, \restrict{V}{H})\) is equivariantly homeomorphic to \(G/{H} \times B_R(V).\)

A \(G\)-cell complex in which all cells are modeled on trivial representations, meaning that all cells are of the form \(G/{H} \times D^{n+1}\), is called a \emph{\(G\)-CW complex}. Any \(G\)-cell complex is \(G\)-homotopy equivalent to a \(G\)-CW complex (see \cite[Proposition 3.3]{borodzik2021khovanov}).


\section{Homotopy coherent diagrams}
\label{sec:hkoh}
Homotopy coherent diagrams represent a relaxed notion of a diagram of topological spaces, one that is required to commute only up to coherent homotopies. In this section we present one approach of introducing group actions on homotopy coherent diagrams.
\subsection{Homotopy coherent diagrams}
In the following, every indexing category \(\mathcal C\) will have the property that for any two objects \(c, d \in \ob(\mathcal C)\), there are only finitely many chains of morphisms starting at \(c\) and ending at \(d\). In particular, the category \(\mathcal C\) does not have any non-identity isomorphisms. This last condition simplifies the formulas defining homotopy coherent diagrams and their homotopy colimits, allowing us to consider only non-identity morphisms (see \cite[Observation 4.12]{Lawson_2020}).

The formalism for homotopy coherent diagrams that we use has been described by \cite{Vogt1973}, already in the case of an arbitrary small indexing category. A slightly different realization, in the case of strictly commutative diagrams, can be found in \cite{MR0232393}, and a different approach was introduced by \cite{MR0365573}.

\begin{df}
  \label{df: hkoh}
  Let \(\mathcal C\) denote a small category and assume that \(\mathcal C\) does not have non-identity isomorphisms. A \emph{homotopy coherent diagram} \(F \colon \mathcal C \to \tops_*\) consists of assignments: to each \(x \in \mathcal C\), a pointed topological space \(F(x) \in \tops_*\), and to each sequence \(x_0 \xto{f_1} x_1 \xto{f_2} \dotsb \xto{f_n} x_n\) of composable non-identity morphisms in \(\mathcal C\), of a pointed continuous map
  \[F(f_n, \dotsc, f_1) \colon [0, 1]^{n-1} \times F(x_0) \to F(x_n).\]
  These maps are required to satisfy the following conditions:
  \begin{itemize}
  \item \(F(f_n, \dotsc, f_{i+1})(t_{i+1}, \dotsc, t_{n-1}) \circ F(f_i, \dotsc, f_1)(t_1, \dotsc, t_{i-1}) = F(f_n, \dotsc, f_1)(t_1, \dotsc, t_{i-1}, 0, t_{i+1}, \dotsc, t_{n-1}),\)
  \item \(F(f_n, \dotsc, f_{i+1} \circ f_i, \dotsc, f_1)(t_1, \dotsc, t_{i-1}, t_{i}, \dotsc, t_{n-1}) = F(f_n, \dotsc, f_1)(t_1, \dotsc, t_{i-1}, 1, t_{i+1}, \dotsc, t_{n-1}).\)
  \end{itemize}
\end{df}

We will sometimes say that a homotopy coherent diagram \(F \colon \mathcal C \to \tops_*\) is \emph{of shape \(\mathcal C\)}.

Denote by \(\mathcal C_{n+1}(x_0, x_n)\) the set of composable chains of non-identity morphisms in \(\mathcal C\) of length \(n\), starting at \(x_0\) and ending at \(x_n\):
\[\mathcal C_{n+1}(x_0, x_n) \coloneqq \{x_0 \xto{f_1} \dotsb \xto{f_n} x_n \suchthat \forall i \enspace f_i \neq \id_{x_i}\}.\]

\begin{df}
  Given a homotopy coherent diagram \(F \colon \mathcal C \to \tops_*\), the homotopy colimit of \(F\) is the pointed topological space 
  \begin{equation}
    \label{eq: hocolim relations}
    \hocolim F = \{\ast\} \disjoint \coprod_{n \geq 0} 
    \coprod_{x_0, x_n \in \ob(\mathcal C)} \mathcal C_{n+1}(x_0, x_n) \times [0, 1]^n \times F(x_0) /{\sim},
  \end{equation}
  with \(\sim\) defined as follows, for \(f_i \colon x_{i-1} \to x_i\), \(t_i \in [0, 1]\), \(p \in F(x_0)\):
  \begin{itemize}
  \item \((f_n, \dotsc, f_1; t_1, \dotsc, t_n; \ast) \sim \ast,\)
  \item \((f_n, \dotsc, f_{i+1}; t_{i+1}, \dotsc, t_n; F(f_i, \dotsc, f_1)(t_1, \dotsc, t_{i-1}; p)) \sim (f_n, \dotsc, f_1; t_1, \dotsc, t_{i-1}, 0, t_{i+1}, \dotsc, t_n; p),\)
  \item \((f_n, \dotsc, f_{i+1} \circ f_i, \dotsc, f_1; t_1, \dotsc, t_{i-1}, t_{i+1}, \dotsc, t_n; p) \sim (f_n, \dotsc, f_1; t_1, \dotsc, t_{i-1}, 1, \dotsc, t_{i+1}, \dotsc, t_n; p)\), \(i < n\),
  \item \((f_{n-1}, \dotsc, f_1; t_1, \dotsc, t_{n-1}; p) \sim (f_n, \dotsc, f_1; t_1, \dotsc, t_{n-1}, 1; p)\).
  \end{itemize}
\end{df}

\begin{df}
  A \emph{homomorphism} of homotopy coherent diagrams \(F_1, F_0 \colon \mathcal C \to \tops_*\) is a collection of (pointed, continuous) maps \(\phi_x \colon F_1(x) \to F_0(x)\) for each \(x \in \ob(\mathcal C)\) such that
  \[F_0(f_n, \dotsc, f_1)(t_1, \dotsc, t_{n-1}) \circ \phi_x = \phi_y \circ F_1(f_n, \dotsc, f_1)(t_1, \dotsc, t_{n-1})\]
  for all sequences \(x \xto{f_1} \dotsb \xto{f_n} y\) of morphisms in \(\mathcal C\) and all \((t_1, \dotsc, t_{n-1}) \in [0, 1]^{n-1} \).
\end{df}

A homomorphism of homotopy coherent diagrams \(\varphi \colon F_1 \to F_2\) induces a continuous map \(\hocolim \varphi \colon \hocolim F_1 \to \hocolim F_2\) in \Cref{eq: hocolim relations}, via 
\[(f_n, \dotsc, f_1; t_1, \dotsc, t_n; p) \mapsto (f_n, \dotsc, f_1; t_1, \dotsc, t_n; \varphi_{x_0}(p)),\]
where \(p \in F(x_0)\) and \(x_0\) is the target of \(f_n\).

A more relaxed notion of map between homotopy coherent diagrams takes the form of a larger homotopy coherent diagram, as in the definition below.

\begin{df}
  A \emph{natural transformation} of homotopy coherent diagrams \(F_1, F_0 \colon \mathcal C \to \tops_*\) is a homotopy coherent diagram \(\eta \colon \two{} \times \mathcal C \to \tops_*\) with
  \(\restrict{\eta}{\{i\} \times \mathcal C} = F_i\)
  for \(i = 0, 1\).
\end{df}

The data of a natural transformation also contains maps \(\eta_x \colon F_1(x) \to F_0(x)\) for \(x \in \mathcal C\). If they are all weak homotopy equivalences, we call \(\eta\) a \emph{weak equivalence} (of homotopy coherent diagrams). However, the \(\eta_x\) do not immediately commute with the identifications in \Cref{eq: hocolim relations}; rather, the comparison map on homotopy colimits is defined up to homotopy, and the following holds.

\begin{prop}(\cite{Vogt1973})
  Let \(F_0, F_1 \colon \mathcal C \to \tops_*\) be homotopy coherent diagrams and \(\eta \colon F_1 \to F_0\) a natural transformation. Then there is a map \(\hocolim \eta \colon \hocolim F_1 \to \hocolim F_0\), well defined up to homotopy. If the components \(\eta_x \colon F_1(x) \to F_0(x)\) are all weak homotopy equivalences, so is \(\hocolim \eta\).
\end{prop}


\subsection{Homotopy coherent diagrams as topological diagrams}
\label{section: hkoh-topo}
\cite{Vogt1973} provides another description of homotopy coherent diagrams, which is of use to us. The following again supposes that \(\mathcal C\) has no non-identity isomorphisms. We work in the category \(\tops_*\) of compactly generated topological spaces.

A \emph{topological category} is a topologically enriched small category; that is, one in which the morphism sets carry a topology and composition is a continuous map. If \(\mathcal D\) is a topological category, a \emph{topological diagram} \(F \colon \mathcal D \to \tops_*\) consists of a function 
\[F_0 \colon \ob(\mathcal D) \to \ob(\tops_*)\]
 along with continuous maps
\[F_{A, B} \colon \mathcal D(A, B)_+ \smashp F_0(A) \to F_0(B)\]
satisfying
\begin{itemize}
\item \(F_{A, B}(\id_A; -) = \id_{F_0(A)}\),
\item for all pairs of composable morphisms \(A \xto{f} C \xto{g} B\),
      \[F_{A, B}(g \circ f; -) = F_{C, B}(g; F_{A, C}(f; -)).\]
\end{itemize}

Given an ordinary small category \(\mathcal C\), one can treat its morphism sets as discrete spaces, whereby topological diagrams are the same as commutative diagrams \(\mathcal C \to \tops_*\). A more interesting construction is the following.

\begin{df}
  The \emph{free topological category} associated to \(\mathcal C\) is the topological category \(\free \mathcal C\) wih \(\ob(\free \mathcal C) = \ob(\mathcal C)\) and morphism spaces
  \[\free \mathcal C(A, B) = \disjoint_{n \geq 0} \mathcal C_{n+1}(A, B) \times I^n /{\sim}\]
  where 
  \(\sim\) is generated by
  \[(f_n, \dotsc, f_0; t_1, \dotsc, t_{i-1}, 1, t_i, \dotsc, t_n) \sim (f_n, \dotsc, f_i \circ f_{i-1},\dotsc, f_0; t_1, \dotsc, t_{i-1}, t_{i+1}, \dotsc, t_n).\]
  Composition in \(\free \mathcal C\) is defined by the formula
  \[(f_n, \dotsc, f_0; t_1, \dotsc, t_n) \circ (g_m, \dotsc, g_0; u_1, \dotsc, u_m) = (g_n, \dotsc, f_0, g_m, \dotsc, g_0; u_1, \dotsc, u_m, 0, t_1, \dotsc, t_n).\]
\end{df}

The assignment is functorial; given a functor of small categories \(\eta \colon \mathcal C \to \mathcal D\), there is an induced continuous functor \(\free \eta \colon \free \mathcal C \to \free \mathcal D\) given by
\[\free \eta (f_n, \dotsc, f_0; t_1, \dotsc, t_n) = (\eta f_n, \dotsc, \eta f_0, t_1, \dotsc, t_n).\]

The free topological category construction allows us to redefine homotopy coherent diagrams as follows.

\begin{df}
  A homotopy coherent diagram \(F \colon \mathcal C \to \tops_*\) is defined as a topological diagram \(F \colon \free \mathcal C \to \tops_*\).
\end{df}

Given a category \(\mathcal C\), let \(\cone{\mathcal C}\) denote the 'cone category over \(\mathcal C\)', that is the category with objects \(\ob(\cone {\mathcal C}) = \ob(\mathcal C) \cup \{\ast\}\), \(\ast\) being made a terminal object of \(\mathcal C\).
Define the topological functor \(\phi_{\bullet} \colon \free \mathcal C \to \tops_*\)  by
\[\phi_\bullet(c) = \free \cone{\mathcal C}(c, \ast).\]
Then the formula for the homotopy colimit in \Cref{eq: hocolim relations} can be restated as follows:
\begin{equation}
\label{eq: steimle hocolim}
\hocolim F = \coeq \left( \coprod_{u, v \in \mathcal C} \phi_\bullet(d) \smashp \free \mathcal C(c, d) \smashp F(c) \rightrightarrows \coprod_{u \in \ob(\mathcal C)} \phi_\bullet(c) \smashp F(c)\right).
\end{equation}

This description has been pointed out by \cite{steimle2011homotopy}, who also defines homotopy colimits using a universal property. Namely, suppose we are given a homotopy coherent diagram \(F \colon \mathcal C \to \tops_*\) and its extension to a homotopy coherent diagram \(G \colon \cone{\mathcal C} \to \tops_*\). Then, any pointed continuous map \(\varphi \colon G(\ast) \to Z\) induces another extension of \(F\) to a homotopy coherent diagram \(G' \colon \cone{\mathcal C} \to \tops_*\) with \(G'(\ast) = Z\). Consider then the category with objects extensions of \(F\) to a coherent diagram \(G \colon \cone{\mathcal C} \to \tops_*\), and morphisms \(G \to G'\) the maps \(\varphi \colon G(\ast) \to G'(\ast)\) such that \(G'\) is induced from \(G\) by \(\varphi\). The \emph{homotopy colimit} of \(F\) can be defined as an initial object in this category; indeed, \cite[Proposition 3.2]{steimle2011homotopy} verifies that \Cref{eq: steimle hocolim} satisfies this universal property.

A natural transformation of homotopy coherent diagrams \(\eta \colon F_1 \transform F_0\) of shape \(\mathcal C\) is then a topological functor \(\eta \colon \free(\two{} \times \mathcal C) \to \tops_*\) with \(F_i = \eta \circ \free(\iota_i)\), where \(\iota_i \colon \mathcal C \to  \two{} \times \mathcal C\) denotes the inclusions for \(i = 1, 0\). Respectively, a homomorphism \(F_1 \to F_2\) is a topological functor \(\theta \colon \two{} \times \free{\mathcal C} \to \tops_*\).


\subsection{External actions on homotopy coherent diagrams}
\label{sec: free topological category}
Equivariant diagrams of spaces have been studied i.e. in \cite{sjack}, \cite{VillarroelFlores2004TheAB}, who consider the notion of an action of a group \(G\) on a diagram \(X \colon I \to \mathcal C\) in the case that \(I\) carries an action by \(G\); \cite{Dotto-Moi} and \cite{Dotto_2016} have further considered properties of homotopy colimits of such functors. We instead require a notion of a ``\(G\)-homotopy coherent diagram'', and one such has been proposed in \cite{stoffregen2018localization}.

Recall that we consider a group \(G\) as a category with one object and morphisms the elements of \(G\), and an action of \(G\) on a category \(\mathcal C\) is a functor 
\(\gamma \colon G \to \cats\)
with \(\gamma(\ast) = \mathcal C\).
Likewise, if \(\mathcal C\) is a topological category, an action of \(G\) on \(\mathcal C\) is a functor
\(\gamma \colon G \to \cats_{\tops}\)
that picks out \(\mathcal C\).

Observe that in the presence of a group action of a group \(G\) on a small category \(\mathcal C\), the free topological category \(\free \mathcal C\) carries an induced action of \(G\) which on objects agrees with the action on \(\mathcal C\) and on morphism spaces is defined by
\[g \acts (f_n, \dotsc, f_0; t_1, \dotsc, t_n) = (g \acts f_n, \dotsc, g \acts f_0; t_1, \dotsc, t_n).\]

The following definition extends \cite[Definition 2.2]{sjack}.

\begin{df}
  \label{df: external action}
  Let \(\mathcal C\) be a category with an action \(\gamma \colon G \to \cats\) and let \(F \colon \mathcal C \to \tops_*\) be a homotopy coherent diagram.

  An \emph{external action} of \(G\) on \(F\) compatible with \(\gamma\) consists of a family of homomorphisms of homotopy coherent diagrams
  \[\{\psi_g \colon F \transform F \circ \gamma_g \suchthat g \in G\}\]
   satisfying:
  \begin{enumerate}
  \item \(\psi_e = \id_F\),
  \item for any \(g, h \in G\), \(\psi_{hg} = (\psi_h, \id_g) \circ \psi_g.\)
  \end{enumerate}
\end{df}

An equivalent definition is \cite[Definition 5.1]{stoffregen2018localization}, which calls for a homomorphism \(\psi \colon G \to \homeo(\vee_{c \in \ob(\mathcal C)}F(c))\) compatible with \(\gamma\) and such that
\[g \acts F(f_n, \dotsc, f_0; t_1, \dotsc, t_n; p) = F(g \acts f_n, \dotsc, g \acts f_0; t_1, \dotsc, t_n; g \acts p).\]

In the presence of an external action, the homotopy colimit realized as the coequalizer of \Cref{eq: steimle hocolim} carries an action of \(G\) by
\(g \acts (t, x) = (gt, gx)\),
as the two defining maps are equivariant. In terms of \Cref{eq: hocolim relations}, the action reads
\begin{equation}
  \label{eq:hocolim-action}
  g \acts (f_n, \dotsc, f_1; t_1, \dotsc, t_n; p) = (g \acts f_n, \dotsc, g \acts f_1; t_1, \dotsc, t_n; \psi_g(p)).
\end{equation}

\begin{df}(see \cite[Definition 5.5]{stoffregen2018localization})
  Suppose that a small category \(\mathcal C\) carries an action of group \(G\). Given two homotopy coherent diagrams \(F_1, F_0 \colon \mathcal C \to \tops_*\) with external actions \(\varphi^1, \varphi^0\) respectively, we call \(F_1\) and \(F_0\) \emph{externally weakly equivalent} if there exists another homotopy coherent diagram \(\eta \colon \two{} \times \mathcal C \to \tops_*\) with external action of \(G\) by \(\tilde{\varphi}\) such that
\((F_i, \varphi^i) = (\restrict{\eta}{\{i\} \times \mathcal C}, \restrict{\tilde{\varphi}}{F_i})\)
for \(i = 1, 0\)
and that the maps \(\eta(1 \to 0, \id_x) \colon F_1(x) \to F_0(x)\) are weak homotopy equivalences.
\end{df}


\subsection{Fixed point diagrams}
Suppose that \(G\) acts on a small category \(\mathcal C\) and  \(H \subseteq G\) is a subgroup. The fixed-point category \(\mathcal C^H\) has objects and morphisms those fixed by \(H\), and the fixed-point category of the topological category \(\free \mathcal C\) is defined analogously. The two topological categories \(\free(\mathcal C^H)\) and \((\free \mathcal C)^H\) are then isomorphic, and given a homotopy coherent diagram \(F \colon \mathcal C \to \tops_*\) with external action of \(G\), the \emph{fixed-point diagram} \(F^H\) is the homotopy coherent diagram \(F^H \colon \mathcal C^H \to \tops_*\) with
\(F^H(c) = F(c)^H\)
and 
\[F^H(f_n, \dotsc, f_0; t_1, \dotsc, t_n) = \restrict{F(f_n, \dotsc, f_0;t_1, \dotsc, t_n)}{F(c)^H}.\]

Moreover, the fixed point set \((\hocolim F)^H\) can be identified as the homotopy colimit of the diagram \(F^H\).

\begin{prop}(\cite[Lemma 5.6]{stoffregen2018localization})
  For any subgroup \(H \subseteq G\) and any h.c. diagram \(F \colon \mathcal C \to \tops_*\) with external action of \(G\) there is a homeomorphism
  \[\hocolim(F^H) \simeq (\hocolim F)^H.\]
\end{prop}


\section{Flow categories}
\label{sec:flowcats}
Flow categories have been introduced by Cohen, Jones and Segal as a way to define stable homotopy types associated to Floer homology (see \cite{cjs-floers}, \cite{cjs-morse}, \cite{cohen2019floer}). They were used by Lipshitz and Sarkar \cite{ls2014} to construct a spatial refinement of Khovanov homology. In this section we describe equivariant cubical flow categories after \cite{borodzik2021khovanov} while supplying an alternative description of the (equivariant) cube flow category using the 'free topological category' construction of the previous section.
\subsection{\(\cornered{n}\)-manifolds and flow categories}
\label{section:cornered}
We reproduce relevant definitions after \cite[Section 3.1]{Lawson_2020}.

\begin{df}(see \cite[Definition 3.1]{Lawson_2020})
  A \emph{\(k\)-dimensional manifold with corners} is a topological space \(X\) equipped with an atlas
  \[\{U_\alpha, \enspace \phi_{\alpha} \colon U_\alpha \to (\real_+)^k\}\] modeled on open subsets of \((\real_+)^k\), with smooth transition functions. For a point \(x\) in a chart \((U, \phi)\), let \(c(x)\) be the number of coordinates of \(\phi(x)\) which are \(0\); \(c(x)\) is independent of the choice of chart. The \emph{codimension-\(i\)} boundary of \(X\) is \(\{x \in X \suchthat c(x) = i\}.\)
  By a \emph{Riemannian metric} on a \(k\)-dimensional manifold with corners \(X\) we mean a Riemannian metric on \(TX.\)
\end{df}

\begin{df}
  A \emph{facet} of \(X\) is the closure of a connected component of the codimension-\(1\) boundary of \(X\). A \emph{multifacet} is a union of disjoint facets of \(X\). A manifold with corners \(X\) is a \emph{multifaceted manifold} if every \(x \in X\) belongs to exactly \(c(x)\) facets of \(X.\) An \emph{\(\cornered{n}\)-manifold} is a multifaceted manifold along with an ordered \(n\)-tuple \((\partial_1X, \dotsc, \partial_nX)\) of multifacets of \(X\) such that:
  \begin{itemize}
  \item \(\setsum_i \partial_iX = \partial X,\)
  \item \(\forall i \neq j \enspace \partial_i X \cap \partial_j X\) is a multifacet of both \(\partial_i X\) and \(\partial_j X.\)
  \end{itemize}
\end{df}
For an \(\cornered{n}\)-manifold \(X\) and an \(\cornered{m}\)-manifold \(Y\), the product space \(X \times Y\) becomes an \(\cornered{n+m}\)-manifold by letting
\[\partial_i(X \times Y) =
  \begin{cases}
    (\partial_i X) \times Y, & 1 \leq i \leq n, \\
    X \times (\partial_{i-n} Y), & n+1 \leq i \leq n+m.
  \end{cases}
\]
If \(X\) is an \(\cornered{n}\)-manifold and \(v \in \{0, 1\}^n\), we write
\[X(v) = \setintersection_{i \colon v_i = 0} \partial_i X, \quad X(\vec{1}) = X.\]
\begin{df}
  
 Let \(X\) and \(Y\) be \(\cornered{n}\)-manifolds; fix a Riemannian metric on \(Y\). A \emph{neat embedding} of \(X\) into \(Y\) is a smooth map \(f \colon X \to Y\) satisfying:
  \begin{itemize}
   \item \(\forall v \in \{0, 1\}^n \enspace f^{-1}(Y(v)) = X(v)\),
   \item \(\forall v \in \{0, 1\}^n \enspace \restrict{f}{X(v)} \colon X(v) \to Y(v)\) is an embedding,
   \item for any pair \(w < v \in \{0, 1\}^n,\) \(f(X(v))\) is perpendicular to \(Y(w)\) with respect to the Riemannian metric on \(Y\).
  \end{itemize}
\end{df}

\begin{df}
  \label{df:fc}
  A \emph{flow category} is a topological category \(\mathcal C\) whose objects \(\ob(\mathcal C)\) form a discrete space, equipped with a grading function \(\gr \colon \ob(\mathcal C) \to \integer,\) and whose morphism spaces satisfy the following conditions:
  \begin{enumerate}[(FC-1)]
  \item for any \(x \in \ob(\mathcal C),\) \(\homs(x, x) = \{\id\},\)
  \item for any \(x, y \in \ob(\mathcal C)\) with \(\gr(x) - \gr(y) = k,\) \(\homs(x, y)\) is a (possibly empty) compact \((k-1)\)-dimensional \(\corner{k-1}\)-manifold,
  \item \label{item:fc3} the composition maps combine to produce diffeomorphisms of \(\corner{k-2}\)-manifolds:
    \[\coprod_{\gr(x) \geq \gr(z) = \gr(y)+i} \homs(z, y) \times \homs(x, z) \equiv \partial_i \homs(x, y).\]
\end{enumerate}
\end{df}
For \(x, y \in \mathcal C,\) the \emph{moduli space} from \(x\) to \(y\) is defined by
\[\mathcal M_{\mathcal C}(x, y) =
  \begin{cases}
    \homs(x, y), & x \neq y, \\
    \emptyset, & x = y.
  \end{cases}\]

Following \cite[Chapter 3]{borodzik2021khovanov}, we now define \emph{equivariant flow categories}.

\begin{df}(\cite[Definition 3.5]{borodzik2021khovanov})
  \label{bps def}
  For \(G\) a finite group, a \(G\)-equivariant flow category is a flow category \((\mathcal C, \gr)\) equipped with the following data:
  \begin{itemize}
  \item for \(g \in G\), a grading-preserving functor \(\mathcal G_g \colon \mathcal C \to \mathcal C\),
  \item a function \(\gr_G \colon \ob(\mathcal C) \to \disjointunion_{H \subseteq G} RO(H),\)
  \end{itemize}
  required to satisfy the following compatibility conditions:
  \begin{enumerate}[(EFC-1)]
  \item \label{bps identity} \(\mathcal G_e\) is the identity functor,
  \item \label{bps group condition} \(\mathcal G_{g} \circ \mathcal G_h = \mathcal G_{g h}\) for all \(g, h \in G\),
  \item \label{bps compatibility condition} \((\mathcal G_g)_{x, y} \colon \modulispace_{\mathcal C}(x, y) \to \modulispace_{\mathcal C}(\mathcal G_g(x), \mathcal G_g(y))\) is a diffeomorphism of \(\langle \gr(x) - \gr(y) - 1\rangle\)-manifolds such that 
    \[\restrict{(G_g)_{x, y}}{\modulispace_{\mathcal C}(z, y) \times \modulispace_{\mathcal C}(x, z)} = (G_g)_{z, y} \times (G_g)_{x, z}\]
    whenever \(z \in \ob(\mathcal C)\), \(\gr(y) < \gr(z) < \gr(x)\),
  \item \(\gr_G(x) \in RO(G_x),\) where \(G_x = \{g \in G \suchthat \mathcal G_g(x) = x\},\)
  \item \(\dim_\real \gr_G(x) = \gr(x),\)
  \item for \(g \in G,\) let \(v_g \colon RO(G_x) \to RO(G_{g \acts x})\) be induced by a map
    \[G_x \ni h \mapsto ghg^{-1} \in g G_x g^{-1} = G_{g \acts x}.\]
    Then we require that for \(g \in G,\) \(x_1, x_2 \in \ob(\mathcal C)\) such that \(\mathcal G_g(x_1) = x_2,\) there be
    \(\gr_G(x_2) = v_g(\gr_G(x_1))\)
    and in particular \(v_g \circ v_h = v_{gh},\)
  \item \label{bps isotropy} for \(x, y \in \ob(\mathcal C)\) define
    \[G_{x, y} = \{g \in G \colon \mathcal G_g(\modulispace_{\mathcal C}(x, y)) \subseteq \modulispace_{\mathcal C}(x, y)\};\]
    then the moduli space
    \(\modulispace_{\mathcal C}(x, y)\)
    is required to be a compact \(G_{x, y}\)-manifold of dimension
    \[\restrict{\gr_G(x)}{G_{x, y}} - \restrict{\gr_G(y)}{G_{x, y}} - \real.\]
  \end{enumerate}
\end{df}

\begin{df}(\cite[Definition 3.6]{borodzik2021khovanov})
  A functor \(f \colon \mathcal C_1 \to \mathcal C_2\) is a \emph{\(G\)-equivariant functor} if:
  \begin{itemize}
  \item \(f\) commutes with group actions on \(\mathcal C_1\) and \(\mathcal C_2,\)
  \item for any object \(x \in \ob(\mathcal C_1)\) there is a \(G_x\)-equivariant map
    \[f_{\gr_G(x)} \colon \gr_G(x) \to \gr_G(f(x))\]
    such that for any \(g \in G\) we have
    \[v_g \circ f_{\gr_G(x)} = f_{\gr_G(\mathcal G_g(x))} \circ v_g.\]
  \end{itemize}
\end{df}

\begin{df}(cf. discussion above \cite[Definition 3.6]{borodzik2021khovanov})
  \label{df: fc-suspension}
  Let \(\mathcal C\) be a \(G\)-equivariant flow category and \(V \in RO(G)\) a virtual representation. The \emph{suspension} of \(\mathcal C\) by \(V\) is the \(G\)-equivariant flow category \(\susp^{V} \mathcal C\) whose objects and morphisms sets, as well as the functors \(\mathcal G_g\) are identical to those of \(\mathcal C\), equipped with the grading function
  \[\gr_G^{\susp^V \mathcal C}(x) = \gr_G^{\mathcal C}(x) + \restrict{V}{G_x} \in RO(G_x).\]
\end{df}

\begin{df}(\cite[Definition 3.8]{borodzik2021khovanov})
  A \(G\)-equivariant functor \(f \colon \mathcal C_1 \to \mathcal C_2\) is called a (trivial) \(G\)-cover if for any \(x, y \in \ob(\mathcal C_1)\) the map
  \[f_{x, y} \colon \modulispace_{\mathcal C_1}(x, y) \to \modulispace_{\mathcal C_2}(f(x), f(y))\]
  is topologically a (trivial) cover and for any object \(x \in \ob(\mathcal C_1),\)
  \(f_{\gr_G(x)}\)
  is an isomorphism of \(G_x\)-representations.
\end{df}

\begin{prop}(\cite[Lemma 3.9]{borodzik2021khovanov})
  \label{bps simplification lemma}
  If \(\mathcal C_2\) is a \(G\)-equivariant flow category, \(\mathcal C_1\) is a flow category, \(f \colon \mathcal C_1 \to \mathcal C_2\) is a trivial cover, and there is an action of \(G\) on \(\mathcal C_1\) satisfying conditions \ref{bps identity}, \ref{bps group condition}, \ref{bps compatibility condition}, such that \(f\) commutes with the action, then there is a unique structure of a \(G\)-equivariant flow category on \(\mathcal C_1\) such that \(f\) is a trivial \(G\)-cover.
\end{prop}


\subsection{Permutohedra and group actions}
\label{sec: geometric permutohedra}

Let \(n \geq 1\) be a natural number. The symmetric group \(\Sigma_n\) acts on \(\real^n\) by
\[\sigma \acts (x_1, \dotsc, x_n) = (x_{\sigma^{-1}(1)}, \dotsc, x_{\sigma^{-1}(n)}).\]

Suppose \(S =(s_1, \dotsc, s_n) \in \real^n\) is an increasing sequence of real numbers. The \emph{\(S\)-permutohedron} \(\permutohedron_S\) is the convex hull of \(\Sigma_n\)-translations of \(S\);
if \([n] = (1, \dotsc, n)\), we write \(\permutohedron_{n-1} = \permutohedron_{[n]}\).

For \(i = 1, \dotsc, n,\) let \(\tau_i = \sum_{j=1}^i s_j.\) If \(P \subseteq S\) is a non-empty subset, consider the half-space
\[H_P = \left\{(x_1, \dotsc, x_n) \in \real^n \suchthat \sum_{i \in P} x_i \geq \tau_{\card{P}}\right\}.\]
Then the permutohedron \(\permutohedron_S\) can equivalently be defined as
\[\permutohedron_S = (\setintersection_{\emptyset \neq P \subsetneq S} H_P) \cap \partial H_S\]
(see e.g. \cite[Theorem 1.5.7]{torict}).
In fact, for any ordered partition \(P_1 \cup \dotsb \cup P_r\) of \(S\), the subset
\[\permutohedron_{P_1, \dotsc, P_r} = \permutohedron_S \cap \setintersection_{i = 1, \dotsc, r-1} \partial H_{P_1 \cup \dotsb \cup P_i}\]
is an \((n-r+1)\)-dimensional face of \(\permutohedron_S.\)
In particular, each \(\permutohedron_{P, S \setminus P}\) is a facet in the sense of \cref{section:cornered}. Moreover, \(\permutohedron_{P_1, \dotsc, P_r}\) is a facet of each of the facets \(\permutohedron_{P_1, \dotsc, P_{i-1} \cup P_i, P_{i+1}, \dotsc, P_r}\), \(i=1, \dotsc, r-1.\) It is the intersection of any two of them; indeed all intersections of the sets \(\permutohedron_{P_1, \dotsc, P_r}\) arise in this way, and the following holds.
\begin{prop}(\cite[Lemma 3.14]{Lawson_2020})
  \label{permutohedron n-mfld}
  The polyhedron \(\permutohedron_S\) becomes an \((n-1)\)-dimensional \(\cornered{n}\)-manifold by letting
  \[\partial_i \permutohedron_S = \coprod_{\card{P} = i} \permutohedron_{P, S \setminus P}.\]
\end{prop}
For any face \(\permutohedron_{P_1, \dotsc, P_{k+1}}\), there are \(2^k\) faces that contain it, all of the form \(\permutohedron_{P_1, \dotsc, P_i \cup \dotsb \cup P_j, \dotsc, P_{k+1}}\); for the sake of this statement, we treat \(\permutohedron_S\) as the 'face' corresponding to the trivial partition. Denote by \(C_{P_1, \dotsc, P_{k+1}}\) the convex hull of barycentra of all the faces containing \(\permutohedron_{P_1, \dotsc, P_{k+1}}\)
\begin{lemma}\cite[cf. Lemma 3.15]{Lawson_2020}
  \label{permutohedron cubical cpx}
  Each of the \(C_{P_1, \dotsc, P_{k+1}}\) is combinatorially equivalent to a \(k\)-dimensional cube, and these cubes form a cubical subdivision of \(\permutohedron_S.\)
\end{lemma}
The poset of faces of \(\permutohedron_S\) is isomorphic to the poset of internal chains in \(\{0, 1\}^n\): to the face \(\permutohedron_{P_1, \dotsc, P_{k+1}}\) one associates the chain \(u^1 > \dotsb > u^k\) with
\[u^j_i =
  \begin{cases}
    1, & s_i \in P_1 \cup \dotsb \cup P_j, \\
    0, & s_i \in P_{j+1} \cup \dotsb \cup P_{k+1}.
  \end{cases}
\]
Consider now the action of a cyclic subgroup \(\langle \sigma \rangle \subseteq \Sigma_n\) on \(\real^n\). The sets \(\permutohedron_S\) and \(\{0, 1\}^n\) are invariant, and so carry the induced action. The poset of faces of \(\permutohedron_S\) also carries this action,
via the corresponding action on the poset of internal chains in \(\{0, 1\}^n\): \[g \acts (u^1 > \dotsb > u^k) = (g \acts u^1 > \dotsb > g \acts u^k).\]
Now, the fixed point poset of the latter action is isomorphic to the poset of internal chains in some other \(\{0, 1\}^{n'}.\) As the \(\langle \sigma \rangle\)-action descends also to the set of barycentra of faces of \(\permutohedron_S\), the following holds.
\begin{prop}
  \label{prop: permutohedra fixed points}
  The fixed-point set \(\permutohedron_S^{\langle \sigma \rangle}\) is a cubical subcomplex, combinatorially equivalent to a lower-dimensional permutohedron.
\end{prop}
Indeed, \cite[Proposition B.18]{borodzik2021khovanov} establishes that this can be realized as a diffeomorphism of \((n'-1)\)-dimensional \(\cornered{n'-1}\)-manifolds.


\subsection{The cube flow category}
\label{sec: cube flow cat}
The free topological category construction applied to a poset \(\mathcal C\) yields a topological category \(\free \mathcal C\) whose morphism spaces are naturally decomposed as cubical complexes. Abouzaid and Blumberg \cite[Section 2.1]{abouzaid2021arnold} used this in the case that \(\mathcal C= \mathcal P\) is a \emph{finite-dimensional poset} in the sense that between for any two elements \(p, q \in \mathcal P\), lengths of chains in \(\mathcal P\) starting at \(p\) and ending at \(q\) form a finite set. The upshot is that moduli spaces of \(\free \mathcal P\) are finite-dimensional and have a well-defined boundary.

Consider the topological category \(\cube(n) \coloneqq \free \two{n}\) as defined in \Cref{sec: free topological category}. We require a description of morphisms spaces of \(\cube(n)\); one such description appears as early as \cite{https://doi.org/10.1112/jlms/s2-9.1.23}, see also \cite{BLOOM20113216}.

Following the definitions in \Cref{sec: free topological category}, given \(u, v \in \two{n}\), the space of morphisms \(\cube(n)(u, v)\) is
\[\disjoint_{m \geq 0} \two{n}(u, v)_m \times [0, 1]^m/{\sim},\]
where \(\two{n}(u, v)_m\) is the set of chains in \(\two{n}\) of length \(m\), lying entirely between \(u\) and \(v\). Here, for \(P \subseteq S\) two chains between \(u\) and \(v\), \(\sim\) identifies the cube \([0, 1]^P\) corresponding to \(P\) to the subset of \([0, 1]^S\) obtained by inserting \(1\)s at \(S \setminus P\)-coordinates.
Thus, \(\cube(n)(u, v)\) is isomorphic as a cubical complex to the permutohedron \(\permutohedron^{u \Delta v}\) as described in \Cref{permutohedron cubical cpx}; here,
\[u \Delta v = \{i \in \{1, \dotsc, n\} \suchthat u_i = 1, v_i = 0\}.\]
Likewise, the composition maps in \(\cube(n)\) recover the multifacets \(\partial \cube(n)(u, v)\) making up the boundary of \(\cube(n)(u, v)\) via
\[\partial_i C(n)(u, v) = \coprod_{\abs{v} - \abs{w} = i} \circ (\cube(n)(v, w) \times \cube(n)(u, v)) \subseteq \partial C(n)(u, v).\]

\begin{df}
  The \emph{cube flow category} is the topological category \(\cube(n) = \free \two{n}\) equipped with the grading function \(\ob(\cube(n)) = \two{n} \to \mathbb Z\) defined by \((u_1, \dotsc, u_n) \mapsto \abs{u} = u_1 + \dotsb + u_n\).
\end{df}

Consider again the \(\mathbb Z_m\)-action on the category \(\two{n'} = \two{nm} \cong (\two{n})^m\), as in \Cref{section: the cube}. The topological category \(\cube(nm)\) carries an action of \(\mathbb Z_m\), by functors \(\gamma_g \colon \cube(nm) \to \cube(nm)\), \(g \in \mathbb Z_m\), and the maps of moduli spaces \((\gamma_g)_{u, v}\) are isomorphisms of cubical complexes. For \(u \in \cube(nm)\), denote by \((\integer_m)_u\) the isotropy group of \(u\) and consider the \((\integer_m)_u\)-representation \(V_u = \real^{u \Delta 0}\).

\begin{prop}(\cite[Proposition 3.10]{borodzik2021khovanov})
  \label{cube action gives cube flow cat}
  The \(\mathbb Z_m\)-action defines a structure of \(\integer_m\)-flow category on \(\cube(n)\), equipped with the grading function \(\gr_{\integer_m}(u) = V_u\).
\end{prop}

In keeping with the conventions of \cite{borodzik2021khovanov}, we denote this category by \(\cube_\sigma(n')\), \(\sigma\) referring to the permutation of \([nm]\) of order \(m\) which defines the action.

\begin{df}(\cite[Definition 3.11]{borodzik2021khovanov})
  \label{bps e-cfc definition}
  A \emph{\(\integer_m\)-equivariant cubical flow category} is a \(\integer_m\)-equivariant flow category equipped with a \(\integer_m\)-cover \(f \colon \susp^V \mathcal C \to C_{\sigma}(n'),\) for \(\sigma\) of order \(m\) in \(\Sigma_{n'},\) and for some \(\integer_m\)-virtual representation \(V.\)
\end{df}


\subsection{Equivariant cubical neat embeddings}
Fix an action of \(G = \integer_m\) on \(\two{n}\) and denote the induced equivariant flow category by \(\cube_\sigma(n)\). Let \(V \in \reps(G)\) be an orthogonal \(G\)-representation, \(u, v \in \ob(\cube_\sigma(n))\), \(u > v\). Denote by \(V_{u, v}\) the restriction of the representation \(V\) to the subgroup \(G_{u, v} = G_u \cap G_v\). Let moreoever \(e_\bullet = (e_0, \dotsc, e_{n-1})\) be a sequence of non-negative integers. Define
\[E(V)_{u, v} = \prod_{i = \abs{v}}^{\abs{u} - 1} B_R(V_{u, v})^{e_i} \times \cube_\sigma(u, v).\]
For any \(g \in G\), there is a map 
\[g \cdot (-) \colon E(V)_{f(x), f(y)} \to E(V)_{f(x), f(y)}\]
taking 
\(\cube_\sigma(n)(u, v)\) to \(\cube_{\sigma}(n)(gu, gv)\)
and
\(V_{u, v}\) to \(V_{gu, gv}\).

\begin{df}(\cite[Definition 3.14]{borodzik2021khovanov}; cf. \cite[Definition 3.25]{Lawson_2020})
  \label{df: equivariant neat embedding}
  Let \((\mathcal C, f \colon \mathcal \susp^VC \to \cube_\sigma(n))\) be an equivariant cubical flow category. An \emph{equivariant cubical neat embedding} of \(\mathcal C\), relative representation \(V \in \reps(G)\) and relative sequence \(e_\bullet = (e_0, \dotsc, e_{n-1}) \in \naturals^n\) is a collection of \(G_{x, y}\)-equivariant neat embeddings \(\iota_{x, y} \colon \modulispace(x, y) \to E(V)_{f(x), f(y)}\) such that:
  \begin{enumerate}[(CNE-1)]
  \item \label{cne1} for all \(x, y \in \ob(C(n))\), the following diagram commutes:
    \begin{equation*}
      \begin{tikzcd}[column sep = huge, row sep = huge]
        \modulispace_{\mathcal C}(x, y) \arrow{r}{\iota_{x, y}} \arrow[swap]{dr}{f}
        & E(V)_{f(x), f(y)} \arrow{d}{\pi_2} \\
        & C_\sigma(n)(f(x), f(y)),
      \end{tikzcd}
    \end{equation*}
  \item \label{cne2} for all \(u, v \in \ob(\cube_\sigma(n)),\) the map
    \[\coprod_{\substack{x, y \in \ob(\mathcal C) \\ f(x) = u, f(y) = v}} \iota_{x, y} \colon \coprod_{\substack{x, y \in \ob(\mathcal C) \\ f(x) = u, f(y) = v}} \modulispace_{\mathcal C}(x, y) \to E(V)_{u, v}\]
    is a neat embedding,
  \item \label{cfc embedding commutativity} for all \(x, y, z \in \ob(\mathcal C),\) the following diagram commutes:
    \begin{equation*}
      \begin{tikzcd}
        \modulispace_{\mathcal C}(y, z) \times \modulispace_{\mathcal C}(x, y) \arrow{r}{\circ} \arrow[swap]{d}{\iota_{y, z} \times \iota_{x, y}}
        & \modulispace_{\mathcal C}(x, z) \arrow{d}{\iota_{x, z}} \\
        E(V)_{f(y), f(z)} \times E(V)_{f(x), f(y)} \arrow{r}{\circ}
        & E(V)_{f(x), f(z)}.
      \end{tikzcd}
    \end{equation*}
  \item \label{cne4} for all \(x, y \in \ob(\mathcal C)\) and all \(g \in G\), the following diagram commutes:
    \begin{equation*}
    \begin{tikzcd}
      \modulispace_{\mathcal C}(x, y) \arrow{r}{\iota_{x, y}} \arrow{d}[swap]{(\mathcal G_g)_{x, y}} &
      E(V)_{f(x), f(y)} \arrow{d}{(g, \gamma_g)} \\
      \modulispace_{\mathcal C}(gx, gy) \arrow{r}{\iota_{gx, gy}} &
      E(V)_{f(gx), f(gy)}
    \end{tikzcd}
  \end{equation*}
  \end{enumerate}
\end{df}

\begin{prop}(\cite[Proposition 3.16]{borodzik2021khovanov})
  Any equivariant cubical flow category admits an equivariant cubical neat embedding.
\end{prop}

In order to define the geometric realization of an equivariant cubical flow category, we need certain extensions of neat embeddings.

\begin{df}(see \cite[Section 3.6]{borodzik2021khovanov}, \cite[Definition 3.25]{Lawson_2020}, \cite[Convention 3.27]{Lawson_2020})
  \label{e-fne}
   An \emph{equivariant framed cubical neat embedding} consists of extensions of the maps \(\iota_{x, y}\) to \(G_{x, y}\)-equivariant maps
   \[\bar{\iota}_{x, y} \colon \prod_{i = \abs{f(y)}}^{\abs{f(x)} - 1} B_\varepsilon(V_{f(x), f(y)})^{e_i} \times \modulispace_{\mathcal C}(x, y) \to E(V)_{f(x), f(y)}\]
   satisfying the conditions analogous to those of \Cref{df: equivariant neat embedding}:
   \begin{enumerate}[(FNE1)]
  \item \label{framed embedding 1} for all \(x, y \in \ob(\cube(n))\), the following diagram commutes:
    \begin{equation*}
      \begin{tikzcd}[column sep = huge, row sep = huge]
        \prod_{i=\abs{f(y)}}^{\abs{f(x)}-1} B_\varepsilon(V)^{e_i} \times \modulispace_{\mathcal C}(x, y) \arrow{r}{\tilde{\iota}_{x, y}} \arrow{d}[swap]{\pi_2} &
        E(V)_{f(x), f(y)} \arrow{d}{\pi_2} \\
        \modulispace_{\mathcal C}(x, y) \arrow{r}{f} &
        \cube(n)(u, v),
      \end{tikzcd}
    \end{equation*}
  \item \label{framed embedding 2} for all \(u > v \in C(n)\) the induced map
  \[\coprod_{f(x) = u, f(y) = v} \bar{\iota}_{x, y} \colon \coprod_{f(x) = u, f(y) = v} \left[\prod_{i=\abs{v}}^{\abs{u}-1} B_\varepsilon(V)^{e_i} \right] \times \modulispace_{\mathcal C}(x, y) \to E(V)_{u, v}\]
  is an embedding,
  \item \label{framed embedding 3} for all \(x, y, z \in \ob(\mathcal C)\), the following diagram commutes:
    \begin{equation*}
      \begin{tikzcd}
        \prod_{i = \abs{f(y)}}^{\abs{f(z)}-1} B_\varepsilon(V)^{e_i} \times \modulispace_{\mathcal C}(y, z) \times \prod_{i = \abs{f(x)}}^{\abs{f(y)}-1} B_\varepsilon(V)^{e_i} \times \modulispace_{\mathcal C}(x, y) \arrow{r}{\Upsilon} \arrow{d}[swap]{\tilde{\iota}_{y, z} \times \tilde{\iota}_{x, y}} & 
        \prod_{i = \abs{f(x)}}^{\abs{f(z)}-1} B_\varepsilon(V)^{e_i} \times \modulispace_{\mathcal C}(x, z) \arrow{d}{\tilde{\iota}_{x, z}} \\
        E(V)_{f(y), f(z)} \times E(V)_{f(x), f(y)} \arrow{r}{\circ} &
        E(V)_{f(x), f(z)},
    \end{tikzcd}
  \end{equation*}
  where \(\Upsilon\) merges the \(\varepsilon\)-terms and applies the composition map in \(\mathcal C\) to the moduli spaces.
  \item \label{framed embedding 4}
    The following diagram commutes:
    \begin{equation*}
    \begin{tikzcd}
      \prod_{i=\abs{f(y)}}^{\abs{f(x)}-1} B_\varepsilon(V)^{e_i} \times \modulispace_{\mathcal C}(x, y) \arrow{r}{\tilde{\iota}_{x, y}} \arrow{d}[swap]{(g, (\mathcal G_g)_{x, y})} &
      E(V)_{f(x), f(y)} \arrow{d}{(g, \gamma_g)} \\
      \prod_{i = \abs{f(y)}}^{\abs{f(x)}-1} B_\varepsilon(V)^{e_i} \times \modulispace_{\mathcal C}(gx, gy) \arrow{r}{\tilde{\iota}_{gx, gy}} &
      E(V)_{f(gx), f(gy)}
    \end{tikzcd}
    \end{equation*}
  \end{enumerate}
\end{df}

\begin{prop}
  Any equivariant cubical neat embedding can be framed, granted \(\varepsilon\) small enough.
\end{prop}
\begin{proof}
  One choice is
  \[\bar{\iota}_{x, y} \colon (t, p) \mapsto (t+\pi_{f(x), f(v)}^R(\iota_{x, y}(p)), \pi_{f(x), f(v)}^M(\iota_{x, y}(p))),\]
  where \(\pi_{u, v}^R \colon E(V)_{u, v} \to \left[\prod_{i = \abs{v}}^{\abs{u} - 1} B_R(V)^{d_i}\right]\), \(\pi_{u, v}^M \colon E(V)_{u, v} \to \cube_\sigma(n)(u, v)\) are the projections. The \(\tilde{\iota}\) thus constructed are equivariant because the \(\iota_{x, y}\), \(\pi^R_{u, v}\) and \(\pi^M_{u, v}\) are.
  The conditions \ref{framed embedding 1}, \ref{framed embedding 3}, \ref{framed embedding 4} also follow from the analogous conditoins \ref{cne1}, \ref{cfc embedding commutativity}, \cref{cne4} placed on \(\iota_{x, y}\). Condition \ref{framed embedding 2} follows from \ref{cne2} together with \ref{cne1}: for \ref{cne1} assures that for all \(p \in \modulispace_{\mathcal C}(x, y)\), \((\pi_{f(x), f(y)}^R)^{-1}(f(p))\) and \(\iota_{x, y}(\modulispace_{\mathcal C}(x, y))\) are transverse in \(E(V)_{f(x), f(y)}\). Hence, for \(\varepsilon\) small enough, the map in \ref{framed embedding 2} is still injective.
\end{proof}


\subsection{Geometric realization of an equivariant cubical flow category}

Let \((\mathcal C, f \colon \susp^V \mathcal C \to \cube_\sigma(n), \iota)\) be an equivariant cubical flow category. Suppose that \(\iota\) has been extended to an equivariant framed cubical embedding. Given \(x \in \ob(\mathcal C)\), write 
\(u = f(x) \in  \two{n}\) and
\[\cell(x) = \prod_{i=0}^{\abs{u} - 1} B_R(V)^{e_i} \times \prod_{i=\abs{u}}^{n-1} B_\varepsilon(V)^{e_i} \times \cube_{\sigma}(n)^+(u, \vec{0}).\]
Here, \(\cube_\sigma(n)^+\) is the topological category \(\free(\two{n}_+)\), so that the morphism spaces are
  \begin{equation}
    \label{cube-plus coherent nerve}
  \cube_\sigma(n)^+(u, \vec{0}) = 
  \begin{cases} 
    \cube_\sigma(n)^+(u, \vec{0}) \times [0, 1], & u \neq 0, \\
    \{0\}, & u = 0. \\
  \end{cases}
  \end{equation}
For any \(x, y \in \ob(\mathcal C)\) with \(f(x) = u > v = f(y)\), the map \(\tilde{\iota}_{x, y}\) furnishes a \(G_{x, y}\)-equivariant embedding
  \[\cell_{x, y} \colon \cell(y) \times \cube_\sigma(n)(x, y) \into \partial \cell(x).\]
  \begin{equation}
    \label{eqn: e-cfc attachment}
    \begin{aligned}
      \cell(y) \times \modulispace_{\mathcal C}(x, y) & \\
      \cong & \prod_{i=0}^{\abs{v} - 1} B_R(V)^{e_i} \times \prod_{i=\abs{u}}^{n-1} B_\varepsilon(V)^{e_i} \times \cube_\sigma(n)^+(v, \vec{0}) \times \left(\prod_{i=\abs{v}}^{\abs{u} - 1} B_\varepsilon(V)^{e_i} \times \modulispace_{\mathcal C}(x, y)\right) \\
      \xinto{\iota_{x, y}} & \prod_{i=0}^{\abs{v} - 1} B_R(V)^{e_i} \times \prod_{i=\abs{u}}^{n-1} B_\varepsilon(V)^{e_i} \times \cube_\sigma(n)^+(v, \vec{0}) \times \left(\prod_{i=\abs{v}}^{\abs{u} - 1} B_R(V)^{e_i} \times \cube_\sigma(n)(x, y)\right) \\
      \into & \partial \cell(x).
    \end{aligned}
  \end{equation}
  The realization \(\cfcrealization{\mathcal C}\) is the CW complex obtained by starting with the basepoint \(\ast\) and attaching cells of increasing gradings \(\abs{x} = \abs{f(x)}\). The attaching map for \(\cell(x)\) sends the image of the map \(\cell_{x, y}\) to \(\cell(y)\) (via the inverse of \(\cell_{x, y}\) composed with projection \(\cell(y) \times \modulispace_{\mathcal C}(x, y) \to \cell(y)\)) and the complement \(\partial \cell(x) \setminus \setsum_{\abs{y} < \abs{x}} \im(\cell_{x, y})\) to \(\ast\).

By \cite[Proposition 3.18]{borodzik2021khovanov}, this produces a \(G\)-cell complex, with cell
\[\mathcal C(x_1) \disjoint \mathcal C(x_2) \disjoint \dotsb \disjoint \mathcal C(x_k) \cong G \times_{G_{x_1}} \left(\prod_{i=0}^{\abs{u}-1} B_R(V)^{e_i} \times \prod_{i = \abs{u}}^{n - 1} B_{\varepsilon}(V)^{e_i} \times B_R(\gr_G(x_1))\right)\]
of type
\((G_x, V^{e_1 + \dotsb + e_{n-1}} \oplus \gr_G(x_1))\)
for \(\{x_1, \dotsc, x_k\}\) an orbit of \(x_1 \in \ob(\mathcal C)\) by the \(G\)-action.

\begin{df}(\cite[Definition 3.19]{borodzik2021khovanov})
  \label{df: e-cfc homotopy type}
  If \((\mathcal C, f \colon \Sigma^W \mathcal C \to \cube_{\sigma}(n))\) is an equivariant cubical flow category with an equivariant cubical neat embedding relative orthogonal \(G\)-representation \(V\) and \((e_1, \dotsc, e_{n-1})\), then the \emph{stable equivariant homotopy type of \(\mathcal C\)} is the formal desuspension
  \[\mathcal X(\mathcal C) = \Sigma^{-W-V^{e_0 + \dotsb + e_{n-1}}} \abs{\abs{\mathcal C}},\]
  where \(\abs{\abs{\mathcal C}}\) is the \(G\)-cell complex described above.
\end{df}

This is considered as an object of the equivariant Spanier-Whitehead category, and the proof of \cite[Theorem 1.2]{borodzik2021khovanov} includes its independence of the choices of \(R, \varepsilon, V, (e_1, \dotsc, e_{n-1}, V)\).


\section{Burnside functors}
\label{sec:burnside}
A Burnside functor is a functor into the Burnside \(2\)-category. After \cite{stoffregen2018localization}, we define a way in which a particular type of homotopy coherent diagram can be described as subordinate to a Burnside functor, in the presence of external group actions on both.
\subsection{The Burnside 2-category}
We reproduce definitions from \cite[Section 4.1]{Lawson_2020}.

\begin{df}
  Let \(X\) and \(Y\) be sets. A \emph{correspondence} from \(X\) to \(Y\) is a set \(A\) together with maps \(s \colon A \to X\), \(t \colon A \to Y.\) \(X\) is then called the source and \(Y\) the target of the correspondence, and \(s\) and \(t\) the source- and target-maps thereof.

  Given correspondences \((A, s_A, t_A)\) from \(X\) to \(Y\) and \((B, s_B, t_B)\) from \(Y\) to \(Z\), the composition \((B, s_B, t_B) \circ (A, s_A, t_A)\) is the correspondence \((C, s, t)\) from \(X\) to \(Z\) given by
  \[C = B \times_Y A = \{(b, a) \in B \times A \suchthat t(a) = s(b)\}, \quad s(b, a) = s_A(a), \quad t(b, a) = t_B(b).\]

  Given correspondences \((A, s_A, t_A)\) and \((B, s_A, t_B)\) from \(X\) to \(Y\), a \emph{morphism of correspondences} from \((A, s_A, t_A)\) to \((B, s_B, t_B)\) is a bijection of sets \(f \colon A \to B\) that commutes with source- and target-maps:
  \[s_A = s_B \circ f, \quad t_A = t_B \circ f.\]
  Composition of morphisms of correspondences is then the usual composition of set maps.
\end{df}

\begin{df}
  The \emph{Burnside category} is the weak 2-category \(\burnside\) of finite sets as objects, correspondences as \(1\)-morphisms and morphisms of correspondences as \(2\)-morphisms.
\end{df}

That \(\burnside\) is a \emph{weak} 2-category means that the identity and associativity axioms hold only up to 2-isomorphism.

We will be working with weak \(2\)-functors from the \(1\)-category \(\twon\) to the Burnside category. These functors are examples of lax \(2\)-functors between weak \(2\)-categories; for a more general definition, see e.g. \cite[Definition 4.2]{Lawson_2020}.

\begin{df}(see \cite[Lemma 4.4]{Lawson_2020}, \cite[Definition 3.3]{stoffregen2018localization})
  Let \(\mathcal C\) denote a small\(1\)-category. A strictly unitary Burnside functor (\emph{Burnside functor} for short in the remainder) \(F \colon \mathcal C \to \mathcal B\) consists of the following data:
  \begin{itemize}
  \item for each object \(v \in \ob(\mathcal C),\) a set \(F(v),\)
  \item for each morphism \(u \xto{A} v\) in \(\mathcal C,\) a correspondence \(F(A)\) from \(F(u)\) to \(F(v),\)
  \item for each pair of morphisms \(u \xto{A} v \xto{B} w\) in \(\mathcal C,\) a map of correspondences
    \[F(A, B) \colon F(B) \times_{F(v)} F(A) \to F(B \circ A).\]

  \end{itemize}
  This data is required to satisfy the following condition: for a triple of morphisms \(u \xto{A} v \xto{B} w \xto{C} x\) in \(\mathcal C,\) the diagram
  \begin{equation*}
    \begin{tikzcd}[row sep = huge, column sep = huge]
      F(C) \times_{F(w)} F(B) \times_{F(v)} F(A) \arrow{r}{\id \times F(A, B)} \arrow[swap]{d}{F(B, C) \times \id}
      & F(C) \times_{F(w)} F(B \circ A) \arrow{d}{F(B \circ A, C)} \\
      F(C \circ B) \times_{F(v)} F(A) \arrow{r}{F(A, C \circ B)}
      & F(C \circ B \circ A)
    \end{tikzcd}
  \end{equation*}
  commutes.
\end{df}

\begin{df}
  A \emph{natural transformation} of Burnside functors \(F_1, F_0 \colon \mathcal C \to \burnside\) consists of another Burnside functor \(J \colon \mathcal C \times \two{} \to \burnside\) such that
  \[\restrict{J}{\mathcal C \times \{1\}} = F_1, \enspace \restrict{J}{\mathcal C \times \{0\}} = F_0.\]
  If moreover for every \(x \in \mathcal C\), the \(1\)-morphism \(J(\id_x \times (1 \to 0))\) is an isomorphism, we call \(J\) a \emph{natural isomorphism}.
\end{df}

In the case of indexing category \(\mathcal C = \two{n}\), the data of a Burnside functor \(\two{n} \to \burnside\) can be simplified as follows.

\begin{lemma}(\cite[Lemma 3.4]{stoffregen2018localization})
  \label{burnside extension}
  Suppose that for any \(u, v, v', w\) in \(\two{n}\) with \(u \geq_1 v, v' \geq_1 w\) there are given: finite sets \(F(v)\), finite correspondences \(F(u, v)\), as well as isomorphisms of correspondences 
  \(F_{u, v, v', w} \colon F(v, w) \circ F(u, v) \to F(v', w) \circ F(u, v')\)
  are given in such a way that:
  \begin{enumerate}[(1)]
  \item \(F_{u, v, v', w} = F^{-1}_{u, v', v, w}\)
  \item \label{burnside extension commutativity} a cube in \(\two{n}\) on the left yields a commutative hexagon of \(2\)-morphisms in \(\mathcal B\) on the right.
    \setlength{\perspective}{2pt}
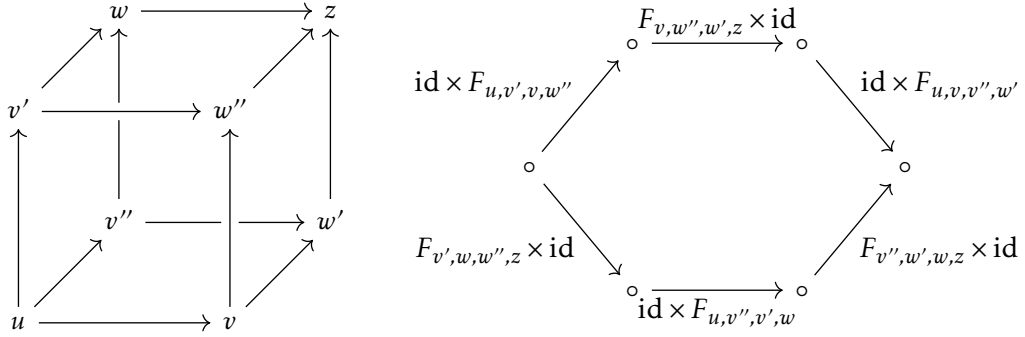
    \[\begin{tikzcd}[row sep={40,between origins}, column sep={40,between origins}]
      &[-\perspective] w \ar{rr} \ar[from=dd] &[\perspective] &[-\perspective] z \\[-\perspective]
    v' \ar{ru} & & w'' \ar[from=ll, crossing over] \ar{ur} \\[\perspective]
      & v'' \ar{rr} & &  w' \ar{uu} \\[-\perspective]
    u \ar{rr} \ar{ur} \ar{uu} && v \ar[crossing over]{uu} \ar{ur}
\end{tikzcd}
    \qquad
    \begin{tikzcd}[column sep ={25}, row sep={35}, every label/.append style = {font = \large}]
      & \circ \ar{rr}{F_{v, w'', w', z} \times \id} 
      && \circ \ar{dr}{\id \times F_{u, v, v'', w'}} & \\
      \circ \ar{dr}[swap]{F_{v', w, w'', z} \times \id} \ar{ur}{\id \times F_{u, v', v, w''}} & & & & \circ \\
      & \circ \ar{rr}[swap]{\id \times F_{u, v'', v', w}}
      && \circ \ar{ur}[swap]{F_{v'', w', w, z} \times \id}
    \end{tikzcd}
    \]
  \end{enumerate}
  Then the data can be extended to a Burnside functor \(F \colon \two{n} \to \burnside\), uniquely up to natural isomorphism, so that \(F_{u, v, v', w} = F^{-1}_{u, v', w} \circ F_{u, v, w}.\)
\end{lemma}


\subsection{External actions on Burnside functors}
\begin{df}(\cite[Definition 3.7]{stoffregen2018localization})
  \label{def sz}
  Fix a Burnside functor \(F \colon \mathcal \mathcal C \to \mathcal B.\) Say there exists an action of \(G\) by \(\psi\) on \(\mathcal C.\) An \emph{external action on \(F\) compatible with \(\psi\)} consists of the following data:
  \begin{enumerate}
  \item \label{sz 1-isos} a collection of \(1\)-isomorphisms
    \[\{\psi_{g, v} \colon F(v) \to F(gv) \suchthat g \in G, \enspace v \in \mathcal C\},\]
  \item \label{sz 2-isos} a collection of \(2\)-isomorphisms
    \[\psi_{g, h, v} \colon \psi_{gh, v} \to \psi_{g, hv} \circ \psi_{h, v}\]
    (note: should such exist, they are unique),
  \item \label{sz 2-composition} for every morphism \(A \colon x \to y\) in \(\mathcal C\) and every \(g \in G,\) a \(2\)-morphism
    \[\psi_{g, A} \colon \psi_{g, y} \circ F(A) \to F(gA) \circ \psi_{g, x}.\]
  \end{enumerate}
  These data are subject to the following conditions:
  \begin{enumerate}[(EB-1)]
  \item \label{sz compatibility with group action} for objects \(u, v \in \mathcal C\) and a morphism \(A \colon u \to v\), the \(2\)-morphism \(\psi_{gh, A}\) is equal to the composite
    \begin{equation*}
    \begin{aligned}
    \psi_{gh, v} \circ F(A) & \xto{\psi_{g, h, v} \circ_1 \id} \psi_{g, hv} \circ \psi_{h, v} \circ F(A) \xto{\id \circ_1 \psi_{h, A}} \psi_{g, hv} \circ F(hA) \circ \psi_{h, u} \\
    & \xto{\psi_{g, hA} \circ_1 \id} F(ghA) \circ \psi_{g, hu} \circ \psi_{h, u} \xto{\id \circ_1 \psi_{g, h, u}} F(ghA) \circ \psi_{gh, u}.
    \end{aligned}
    \end{equation*}
  \item \label{sz compatibility with composition} given a composable pair \(u \xto{A} v \xto{B} w\) in \(\mathcal C\), the \(2\)-morphisms
    \begin{align*}
    \psi_{g, w} \circ F(B) \circ F(A) & \xto{\psi_{g, B} \circ_1 \id} F(gB) \circ \psi_{g, v} \circ F(A) \xto{\id \circ_1 \psi_{g, A}} F(gB) \circ F(gA) \circ \psi_{g, u} \\
    & \xto{F_{gA, gB} \circ_1 \id} F(gB \circ gA) \circ \psi_{g, u}
    \end{align*}
    and 
    \[\psi_{g, w} \circ F(B) \circ F(A) \xto{\id \circ_1 F_{A, B}} \psi_{g, w} \circ F(B \circ A) \xto{\psi_{g, B \circ A}} F(gB \circ gA) \circ \psi_{g, u}\]
    are equal.
  \end{enumerate}
\end{df}

Suppose \(\mathcal C\) carries a \(G\)-action and \(G\) is understood to be acting trivially on \(\two{}\). Then \(\mathcal C \times \{i\}\), \(i = 1, 0\), are \(G\)-invariant subcategories of \(\mathcal C \times \two{}\). Given a Burnside functor \(J \colon \mathcal C \times \two{} \to \burnside\) with an external \(G\)-action by \(\psi\), the subfunctors \(F_i \colon \mathcal C \times \{i\} \to \burnside\) carry external actions induced from \(\psi\) by restriction. This informs the following.

\begin{df}
  \label{sz external equivalence}
  Let \(\mathcal C\) be a small category, acted upon by a group \(G\), and let \(F_1, F_0 \colon \mathcal C \to \burnside\) be functors equipped with external actions of \(G\) by \(\psi_1, \psi_0\). We say that \(F_1\) and \(F_0\) are \emph{equivariantly naturally isomorphic} if there is a natural isomorphism \(J \colon \mathcal C \times \two{} \to \burnside\) between \(F_1\) and \(F_0\), equipped with an external action of \(G\) extending \(\psi_1\) and \(\psi_0\), respectively.
\end{df}

\begin{example}
  Let \(F \colon \two{} \to \burnside\) be a Burnside functor and \(G\) a group, understood to be acting trivially on \(\two{}\); denote the single correspondence in \(F(1, 0)\) by \((A, \enspace s \colon A \to X = F(1), \enspace t \colon A \to Y = F(0))\). Then an external action of \(G\) on \(F\) consists of:
  \begin{enumerate}
  \item invertible correspondences \(\psi_{g, 1} \colon X \to X\), \(\psi_{g, 0} \colon Y \to Y\), for all \(g \in G\),
  \item isomorphisms of correspondences \(\psi_{g, h, v} \colon \psi_{gh, v} \to \psi_{g, v} \circ \psi_{h, v}\) for \(v = 1, 0\) and \(g, h \in G\), carrying no additional information beyond their existence,
  \item isomorphisms of correspondences \(\psi_{g, A} \colon \psi_{g, 0} \circ  A \to A \circ \psi_{g, 1}\) for all \(g \in G\).
  \end{enumerate}
  These are required to satisfy, for all \(g, h \in G\),
  \[\psi_{gh, A} = \psi_{g, h, 0}\circ \psi_{g, A} \circ \psi_{h, A} \circ \psi_{g, h, 1}.\]
\end{example}

By an extension of the \cref{burnside extension}, external actions on Burnside functors from the cube are determined by lower-dimensional data.
\begin{lemma}(\cite[Lemma 3.10]{stoffregen2018localization})
  \label{external burnside simplification}
  Consider the cyclic action of \(\mathbb Z_m\) on \((\two{n})^m\). Suppose \(F \colon (\two{n})^m \to \burnside\) is defined as in \cref{burnside extension}, and that in addition we are given:
  \begin{enumerate}[(1)]
  \item for \(v \in (\two{n})^m\), a \(1\)-isomorphism \(\psi_{g, v} \colon F(v) \to F(gv)\),
  \item for \(g, h \in G\) and \(v \in (\two{n})^m\), a \(2\)-morphism 
    \[\alpha_{g, h, v} \colon \psi_{g, h, v} \to \psi_{g, hv} \circ \psi_{h, v},\]
  \item for each \(g \in \mathbb Z_p\) and \(u \geq_1 v \in (\two{n})^m,\) a \(2\)-morphism
    \[\psi_{g, u, v} \colon \psi_{g, v} \circ F(u, v) \to F(gu, gv) \circ \psi_{g, u}.\]
  \end{enumerate}
  Suppose moreover that this data satisfies the following:
  \begin{enumerate}[(E-1)']
  \item For any \(u \geq_1 v\) and all \(g, h \in G\), 
    \[\psi_{gh, u, v} = \alpha^{-1}_{g, h, u} \circ_2 (\psi_{g, hu, hv} \circ \id) \circ_2 (\id \circ \psi_{h, u, v}) \circ_2 \alpha_{g, h, v},\]
  \item for any \(u \geq_1 v, v' \geq_1 w\) and any \(g \in G\), there is a commutativity hexagon yielding
    \[(F(gu, gv, gv', gw) \circ_2 \id) \circ (\id \circ \psi{g, u, v}) \circ (\psi_{g, v, w} \circ \id) = (\id \circ \psi_{g, u, v'}) \circ (\psi_{g, v', w} \circ \id) \circ (\id \circ F(u, v, v', w)).\]
  \end{enumerate}
  Then there exists a Burnside functor \(F \colon (\two{n})^m \to \burnside\) admitting an external \(\mathbb Z_m\)-action, uniquely up to \(\mathbb Z_m\)-equivariant isomorphism.
\end{lemma}


\section{Spatial refinements}
\label{sec:sp-ref}
The aim of this chapter is to, given a Burnside functor (with external action), produce homotopy coherent diagrams with the property that in the homotopy colimit, vertices of the diagram correspond to cells and the Burnside functor describes degrees of attaching maps.
\subsection{Stars and star maps}
The constructions presented in this sections realise a version of the ``charge map'', associating to a configuration of points in \(\real^n\) a map of spheres \(S^n \to S^n\) (see \cite[Section 1]{segal-conf}). The approach taken here (after \cite{Lawson_2020}) allows for composing such maps between (wedges of) spheres along a Burnside functor, in the end furnishing a homotopy coherent diagram. This was already done equivariantly in \cite[Section 4.4]{stoffregen2018localization}, using spaces of little disks; for our purposes, a wider family of shapes must be used, contatining both disks and products of disks.

In the scope of this section, \(V\) will denote an orthogonal representation of a group \(G\).

For \(X\) a finite \(G\)-set, the \emph{configuration space} of points of \(X\) in a \(G\)-space \(Y\) is the space
\[\conf_X(Y) = \{\{p_x\}_{x \in X} \in Y^k \suchthat \forall x \neq y \in X \enspace p_x \neq p_y\}.\]
Equivalently, a configuration can be seen as an embedding \(f \colon X \to Y\), whereby \(\conf_X(Y)\) carries an action of \(G\) by 
\[(g \acts f)(x) = g \acts f(g^{-1} \acts x).\]
We aim to describe one of the possible equivariant versions of the Pontryagin-Thom collapse map, associating a map between spheres to a configuration. This entails replacing the points of a configuration by ``little stars'', as expanded upon below.

Let \(\sphere(V)\) denote the unit sphere in \(V\). Let \(f \colon \sphere(V) \to \real_+\) be a continuous map. By a \emph{star} in \(V\) we will mean the set
\[B(p, f) = p + \{\alpha \cdot v \in V \suchthat v \in V, \enspace \abs{\alpha} \leq (v)\} \subseteq V\]
for some \(f\) as above and \(p \in V\). The point \(p\) is then called the center point of \(B(p, f)\), and a star is understood to come with a distinguished center point. If \(f\) is \(G\)-invariant, we  call \(B(0, f)\) an \emph{invariant star} in \(V\). Any star is a star-shaped subset of \(V\); an invariant star in \(V\) is a \(G\)-invariant subset of \(V\), and as such becomes a \(G\)-space.

  Let \(\mathbb A \colon \two{} \to \burnside\) be a Burnside functor with external action by \(G\), subordinate to the trivial action of \(G\) on \(\two{}\). Per \cref{def musyt}, this consists of a correspondence \(X \xot{s} A \xto{t} Y\) along with actions \(\phi_{g, X} \colon X \to X\), \(\phi_{g, Y} \colon Y \to Y\) and \(\phi_{g, A} \colon A \to A\), which commute with the source- and target-maps.

  Suppose \(\{B_x\}_{x \in X}\) is a set of \(G\)-invariant stars in \(V\). We fix the radial homeomorphisms \sloppy \({\varphi_{x, y} \colon B_x \to B_y}\) for all \(x, y\); these satisfy \(\varphi_{x, z} = \varphi_{y, z} \circ \varphi_{x, y}\). The space \(B(X, V) \coloneqq \disjoint_{x \in X} B_x\) carries an action of \(G\) by \(g \acts (x, v) = (g \acts x, g \acts \phi_{x, g \acts x}(v)).\)
  The space \(\conf(\{B_x\}, s)\) is defined as the space of embeddings \(\gamma \colon A \to B(X, V)\) satisfying \(\gamma(a) \in B_{s(a)}.\) Topologically, this is the same as \(\disjoint_{x \in X} \conf_{s^{-1}(x)}(B_x).\) The space \(\conf(\{B_x\}, s)\) carries an action of \(G\) by
  \[(g \acts \gamma)(a) = g \acts \varphi_{g, g \acts x} \gamma(g^{-1} \acts a).\]

Replacing the points of \(A\) by little stars, we consider the space \(\stars(\{B_x\}, s)\) of embeddings of \(A\)-labeled stars in \(B(X, V)\), again with \(B_a \subseteq B_{s(a)}\). This is topologized as a subset of \sloppy \({\conf(\{B_x\}, s) \times \maps(X \times S(V), \real^+)}\), and again carries an action by 
  \[(g \acts \gamma)(a, v) = g \acts \varphi_{g, g \acts x} \gamma(g^{-1} \acts a, v).\]

A configuration of stars with centers \(\{p_x\}_{x \in X}\) can be deformed to one whose stars are all spheres with radius
\[\frac{1}{3}\min_{x, y \in X}(d(p_x, p_y), d(p_x, V \setminus B)).\]
This establishes a strong equivariant deformation retraction from \(\stars(\{B_x\}, s)\) to a bundle of points over \(\conf(\{B_x\}, s)\). The map is also equivariant and induces homotopy equivalences of fixed-point sets, implying the following.

\begin{lemma}
  \label{lemma: stars-connectivity}
  The spaces \(\stars(\{B_x\}, s)\) and \(\conf(\{B_x\}, s)\) are \(G\)-homotopy equivalent.
\end{lemma}

Note that the target map of the correspondence played no role in the definition of \(\stars(\{B_x\}_X, s)\); rather, it becomes relevant in the definition of the associated map of spheres. Let \(S^V\) denote the one-point compactification \(V \cup \{\infty\}\), considered as a \(G\)-space.
\begin{df}(\cite[Definition 5.8]{Lawson_2020})
  Let \(\mathbb A = (A, s \colon A \to X, t \colon A \to Y)\) be a correspondence and 
  \[e = \{B_a \subseteq B_{s(a)} \suchthat a \in A\} \in \stars(\{B_x\}, s)\]
  a collection of substars in \(V\).
  Define a map \(\Phi(e, \mathbb A) \colon \vee_{x \in X} S^V \to \vee_{y \in Y} S^V\) on the summand \(S_x^V\) corresponding to \(x \in X\) to be the map of spheres
  \[\restrict{\Phi(e, \mathbb A)}{S_x^V} \colon S_x^V = B_x/{\partial B_x} \to B_x / (B_x \setminus (\bigcup_{\substack{a \in A \\ s(a) = x}} \interior{B}_a)) = \bigvee_{\substack{a \in A \\ s(a) = x}} B_a/{\partial B_a} = \bigvee_{\substack{a \in A \\ s(a) = x}} S_a^V \to \bigvee_{y \in Y} S_y^V,\]
  the last map sending \(S_a^V\) by the identity map to \(S_{t(a)}^V.\)
  Any map \(\vee_{x \in X} S^V \to \vee_{y \in Y} S^V\) that is of the form \(\Phi(e, \mathbb A)\) for some \(e \in \stars(\{B_x\}, s)\), is called a (\(V\)-dimensional) \emph{star map refining the correspondence \(\mathbb A\)}.
\end{df}

We end this section by stating some facts about star maps without proof.

\begin{lemma}
  The map 
  \[\Phi(-, A) \colon \stars(\{B_x\}, s) \to \maps(\vee_{x \in X} S^k, \vee_{y \in Y} S^V)\]
 is continuous. 
\end{lemma}

\begin{lemma}(\cite[Lemma 4.12]{stoffregen2018localization})
  The star map \(\Phi(e, \mathbb A)\) associated to an element \(e \in \stars(\{B_x\}, s)^H\) fixed by subgroup \(H \subseteq G\), is \(H\)-equivariant.
\end{lemma}

\begin{lemma}
  \label{star suspension}
  Given a correspondence \(\mathbb A \colon X \to Y\) as above, \(G\)-representations \(V\) and \(W\), a family of invariant stars in \(\{B_x\}_{x \in X} \subseteq V\) and an invariant star \(B' \subseteq W\), consider the map \(\psi_{B'} \colon \stars(\{B_x\}, s) \to \stars(\{B_x \times B'\}_X, s)\) obtained by taking products of all stars with \(B'\). Then for any \(e \in \stars(\{B_x\}, s)\), the assignment
  \[\Phi(e, \mathbb A) \circ \psi_{B'} \in \maps(\vee_{x \in X} S^{V \oplus W}, \vee_{y \in Y} S^{V \oplus W})\]
  is a \((V \oplus W)\)-dimensional star map refining the correspondence \(\mathbb A\).
\end{lemma}

\begin{lemma}(\cite[Lemma 4.8]{stoffregen2018localization})
  Let \(X \xto{\mathbb A} Y \xto{\mathbb B} Z\) be finite correspondences. Given \(e \in \stars(\{B_x\}, s_A)\) and \(f \in \stars(\{B_y\}, s_B)\), there is a unique arrangement of stars
  \[f \circ e \in \stars(\{B_x\}, s_{B \circ A}, \mathbb B \circ \mathbb A, X) = \stars(\{B_x\}, s_{\mathbb B \circ \mathbb A})\]
  such that there is an equality of star maps
  \(\Phi(f \circ e, \mathbb B \circ \mathbb A) = \Phi(f, \mathbb B) \circ \Phi(e, \mathbb A).\)
  Moreover, the assignment
  \[\circ \colon \stars(\{B_y\}, s_B) \times \stars(\{B_x\}, s_A) \to \stars(\{B_x\}, s_{B \circ A}, B \circ A, X)\]
  is continuous and surjective.
\end{lemma}
\begin{proof}
  For \((b, a) \in B \times_Y A\), \(b \in B\), \(a \in A\), consider the corresponding stars \(e_b \colon B_b \to B_{s_B(b)}\), \(e_a \colon B_a = B_{s_B(b)}\to B_{s_A(a)}\). The substar \(e_{b, a} \colon B_b \to B_{s_A(a)} = B_{s_{B \circ A}(b)}\) is \(e_a \circ e_b.\)
\end{proof}


\subsection{Equivariant spatial refinements}

\begin{df}
  For any finite-dimensional real \(G\)-representation \(V\), consider the set \(\stars(V)\) of stars in \(V\) invariant under the \(G\)-action; write 
  \[\stars(G) = \coprod_{V \in \reps(G)} \stars(V).\]
\end{df}

\begin{lemma}(\cite[Lemma 4.11]{stoffregen2018localization})
  \label{many H-equivariant stars}
  Let \(V\) be a \(G\)-representation, \(s \colon A \to X\) a function, and \(H \subseteq G\) a subgroup. Then, for any integer \(N > 0\), there exists a finite-dimensional representation \(V_N\) such that if \(V\) is another finite-dimensional representation admitting an embedding \(V_N \into V\), then for any family \(\{B_x\}_{x \in X}\) of stars in \(V\), the fixed-point set of \(\stars(\{B_x\}, s)\) under the action of \(H\), denoted by \(\stars(\{B_x\}, s)^H\), is \(N\)-connected (and nonempty).
\end{lemma}

\begin{df}
  Let \(\mathcal C\) be a poset, \(F \colon \mathcal C \to \burnside\) a Burnside functor, \(\tilde{F} \colon \mathcal C \to \tops_*\) a homotopy coherent diagram, and \(V\) an inner product space. We say that \(\tilde{F}\) is a \emph{spatial refinement of \(F\) modeled on \(V\)} if its components are of the form:
  \begin{itemize}
    \item for \(u \in \ob(\mathcal C)\), there are stars \(\{B_x\}_{x \in F(u)} \subseteq \stars(V)\) with
      \[\tilde{F}(u) = \vee_{x \in F(u)} S^V = \disjoint_{x \in F(u)} B_x/{\partial},\]
    \item for \(u, v \in \ob(\mathcal C)\), the component
      \[\tilde{F}(u, v) \colon \free{\mathcal C}(u, v) \to \tops_*(\vee_{x \in F(u)} S^V, \vee_{x \in F(v)} S^V)\]
      equals \(\Phi(-, F(u, v)) \circ \tilde{F}_{u, v}\), where 
      \(\tilde{F}_{u, v} \colon \free{\mathcal C}(u, v) \to \stars(\{B_x\}_{x \in F(u)}, s_{F(u, v)})\)
      is a continuous family of star arrangements.
  \end{itemize}
\end{df}

\begin{df}
  \label{g action on sets of boxes}
  Let \(F \colon \mathcal C \to \burnside\) be a Burnside functor equipped with an external action of \(G\) by \(\psi\). Let \(V\) be a \(G\)-representation and \(\tilde{F} \colon \mathcal C \to \tops_*\) a spatial refinement modeled on \(V\). 
  The spatial refinement \(\tilde{F}\) of \(F\) is called a \emph{\(G\)-coherent refinement modeled on \(V\)} if for all \(g \in G\), \(u, v \in \ob(\mathcal C)\), \(x \in F(u)\), \(t \in \free \mathcal C(u, v)\) and \(p \in B_x/{\partial B_x}\) the equality
  \begin{equation}
  \label{g-coherent refinement equation}
  g \acts \tilde{F}(u, v)(t)(p) = \tilde{F}(g \acts u, g \acts v)(t)(g \acts p)
  \end{equation}
  holds (here, \(g \acts p \in B_{g \acts x}/{\partial B_{g \acts x}}\)).
\end{df}

\begin{prop}(\cite[Proposition 5.11]{stoffregen2018localization})
  \label{e-spatial refinements exist}
  Let \(\mathcal C\) be a small category of length \(n\), equipped with a \(G\)-action. Let \(F \colon \mathcal C \to \burnside\) be a Burnside functor, equipped with an external action of \(G\).
  \begin{enumerate}
  \item \label{e refinements existence item} There exists a finite-dimensional \(G\)-representation \(W\) such that for all finite-dimensional \(G\)-representations \(V\) which admit an embedding of \(W\), there exists a \(G\)-coherent refinement of \(F\) modeled on \(V\).
  \item \label{e refinements equivalence item} There exists a finite-dimensional \(G\)-representation \(W\) such that for all finite-dimensional \(G\)-representations \(V\) which admit an embedding of \(W\), any two \(G\)-coherent refinements of \(F\) modeled on \(V\) are weakly equivalent.
  \item If \(\tilde{F}_V\) is a \(G\)-coherent refinement of \(F\) modeled on \(V\), then for any \(G\)-representation \(V'\), the result of suspending each \(\tilde{F}_V(u)\) and \(\tilde{F}_V(f_n, \dotsc, f_1)\) by \(V'\) gives a \(G\)-coherent spatial refinement of \(F\) modeled on \(V \oplus V'\).
  \end{enumerate}
\end{prop}

The preceding proposition allows us finally to define stable Burnside functors and their geometric realizations. The construction uses homotopy colimits over a slightly larger category.

\begin{df}
  \label{df: double zero}
  The category \(\two{n}_+\) has objects \(\ob(\two{n}_+) = \ob(\two{n}) \cup \{\ast\}\) and morphisms
  \[\homs_{\two{n}_+}(u, v) = 
    \begin{cases}
      \homs_{\two{n}}, & v \in \ob(\two{n}) \\
      \{\ast\}, & v = \ast, u \in \ob(\two{n}), \\
      \emptyset, & v = \ast, u = \vec{0}.
    \end{cases}
  \]
\end{df}

That is, \(\two{n}_+\) can be seen as a \(\two{n}\) with the terminal object \(\vec{0}\) ``doubled''.

\begin{df}
  \label{def stable realization}
  A \emph{stable Burnside functor with external action of a group \(G\)} is a triple 
  \[(F \colon \two{n} \to \burnside, \psi \acts F, V \in \representationring(G))\]
  of Burnside functor \(F\), external action \(\psi\) of \(G\) on \(F\) compatible with an action of \(G\) on \(\mathcal C\), and a virtual representation \(V\).

  The \emph{stable homotopy type} of \((F, \psi, V)\) is the equivariant suspension spectrum (seen as an element of the \(G\)-Spanier-Whitehead category)
  \[\abs{F} \coloneqq \susp^{V-W} \susp^\infty \hocolim \tilde{F}_W^+,\]
  where \(W\) is an orthogonal representation of \(G\) for which \cref{e refinements equivalence item} of \cref{e-spatial refinements exist} holds, \(\tilde{F}_W\) is a spatial refinement of \(F\) modeled on \(W\), and \(\tilde{F}_W^+\) is its extension to a homotopy coherent diagram
  \[\tilde{F}_W^+ \colon \two{n}_+ \to \tops_*\]
  obtained by letting \(\tilde{F}_W^+(\ast)\) be the basepoint.
\end{df}


\section{\(G\)-cubical categories are external actions on Burnside functors}
\label{sec:g-cats-external-actions}
Independently of \cite{stoffregen2018localization}, a notion of external action on a Burnside functor was introduced by Musyt \cite{musyt}. We use this as a stepping stone in order to pronounce the comparison map establishing equivalence between the notions of equivariant cubical flow categories and external actions on Burnside functors (in the sense of \cite{stoffregen2018localization}).
\label{bps-musyt-sz}
\subsection{Musyt's formalism}
\begin{df}
  \label{def musyt}
  Fix a Burnside functor \(F \colon \mathcal \mathcal \two{n} \to \mathcal B.\) Say there exists an action of \(G\) by \(\phi\) on \(\two{n}.\) An \emph{external action on \(F\) compatible with \(\phi\)} consists of the following data:
  \begin{enumerate}
  \item a collection of bijections
    \[\{\phi_{g, v} \colon F(v) \to F(gv) \suchthat g \in G, \enspace v \in \two{n}\},\]
  \item for every pair \(u \geq v\) in \(\two{n}\) and every \(g \in G,\) a bijection
    \[\phi_{g, u, v} \colon F(u, v) \to F(gu, gv).\]
  \end{enumerate}
  These data are subject to the following conditions:
  \begin{enumerate}[(MD-1)]
  \item \label{musyt identity} for any \(u \geq v \in \two{n},\) the maps \(\phi_{e, u} \colon F(u) \to F(u)\) and \(\phi_{e, u, v} \colon F(u, v) \to F(u, v)\) are the identity,
  \item \label{musyt 1-composition} \(\phi_{gh, u} = \phi_{g, hu} \circ \phi_{h, u},\)
  \item \label{musyt 2-composition} \(\phi_{gh, u, v} = \phi_{g, hu, hv} \circ \phi_{h, u, v},\)
  \item \label{musyt source target} \(\phi_{g, u} \circ s = s \circ \phi_{g, u, v},\) \(\phi_{g, v} \circ t = t \circ \phi_{g, u, v},\)
  \item \label{musyt composition} \(\phi_{g, u, w} \circ F(u, v, w) = F(gu, gv, gw) \circ \left(\phi_{g, v, w} \times \phi_{g, u, v}\right)\) is an equality of functions \sloppy \({F(u, v) \times_{F(v)} F(v, w) \to F(gu, gw)}.\)
  \end{enumerate}
  We will refer to the functions \(\phi_{g, v}\) and \(\phi_{g, u, v}\) satisfying \ref{musyt identity}--\ref{musyt composition} as \emph{Musyt data} (of external action on \(F\)).
\end{df}


\subsection{Musyt and SZ}
Essentially, Musyt's formalism corresponds to that of Stoffregen-Zhang by exchanging \(1\)-isomorphisms in \(\mathcal B\) for bijections of sets, and some of the \(2\)-isomorphisms for equalities of functions. We write out the comparison more explicitly.

\begin{construction}
  \label{musyt to sz}
  From Musyt data of external action we produce Stoffregen-Zhang external action as follows.
  \begin{enumerate}
  \item From a bijection \(\phi_{g, u} \colon F(u) \to F(gu),\) we produce a correspondence \(\psi_{g, u} \colon F(u) \to F(gu)\) by
    \[\psi_{g, u} \coloneqq F(u) \xot{\id} F(u) \xto{\phi_{g, u}} F(gu).\]
  \item The 2-morphisms \(\psi_{g, h, u} \colon \psi_{gh, u} \to \psi_{g, hu} \circ \psi_{h, u}\) are given by the map
    \[F(u) \to F(u) \times_{F(hu)} F(hu), \quad a \mapsto \left(a, \phi_{h, u}(a)\right).\]
    They are \(2\)-morphisms because \(\phi_{gh, u} = \phi_{g, hu} \circ \phi_{h, u}.\)
  \item From bijection \(\phi_{g, u, v} \colon F(u, v) \to F(gu, gv)\) we produce a \(2\)-morphism \(\psi_{g, u, v} \colon \psi_{g, v} \circ F(u, v) \to F(gu, gv) \circ \psi_{g, u}\) by the following formula:
    \[[a, t_{F(u, v)}(a)] \mapsto [\phi^{-1}_{g, u}(s_{F(gu, gv)}(\phi_{g, u, v}(a))), \phi_{g, u, v}(a)] = [s_{F(u, v)}(\phi_{g, u, v}(a)), \phi_{g, u, v}(a)],\]
    where \(a \in F(u, v)\) uniquely determines an element of the correspondence \(F(u, v) \times_{F(v)} \psi_{g, v},\) and analogously \(\phi_{g, u, v}(a)\) for \(\psi_{g, u} \times_{F(gu)} F(gu, gv).\)
    The conditions for this \(\psi_{g, u, v}\) to be a \(2\)-morphism are
    \begin{itemize}
    \item \(s_{F(u, v)} = \phi_{g, u}^{-1} \circ s_{F(gu, gv)} \circ \psi_{g, u, v},\) or equivalently \(\phi_{g, u} \circ s_{F(u, v)} = s_{F(gu, gv)} \circ \psi_{g, u, v},\)
    \item \(\phi_{g, v} \circ t_{F(u, v)} = t_{F(gu, gv)} \circ \psi_{g, u, v},\)
    \end{itemize}
    which is exactly the condition \ref{musyt source target} in \cref{def musyt}.
  \end{enumerate}
\end{construction}  

\begin{lemma}
    The \(1\)-morphisms \(\psi_g, u\) and the \(2\)-morphims \(\psi_{g, h, u}\) and \(\psi_{g, u, v}\) satisfy  compatibility conditions \ref{sz compatibility with group action} and \ref{sz compatibility with composition} of \cref{def sz}.
\end{lemma}
\begin{proof}
For \ref{sz compatibility with group action}, take an element \((a, t_{F(u, v)}(a)) \in F(u, v) \times_{F(v)} \psi_{gh, v}.\) The sequence of maps in \ref{sz compatibility with group action} then takes the form
  \begin{align*}
      \left(a, t_{F(u, v)}(a)\right) & \mapsto \left(a, t_{F(u, v)}(a), \phi_{h, v}(t_{F(u, v)}(a))\right) \\ & \mapsto \left(s_{F(u, v)}(a), \phi_{h, u, v}(a), \phi_{h, v}(t_{F(u, v)}(a))\right) \\ & \mapsto \left(s_{F(u, v)}(a), s_{F(hu, hv)}(\phi_{h, u, v}(a)), \phi_{g, hu, hv}(\phi_{h, u, v}(a))\right) \\ & = \left(s_{F(u, v)}(a), s_{F(hu, hv)} (\phi_{h, u, v}(a)), \phi_{gh, u, v}(a)\right) \\ & \mapsto \left(s_{F(u, v)}(a), \phi_{gh, u, v}(a)\right),
  \end{align*}
    which is the same as \(\phi_{gh, u, v},\) as required.

  For \ref{sz compatibility with composition}, we get the following diagram:
  \begin{equation*}
    \begin{tikzcd}[column sep=tiny]
        & \left(a, s_{F(v, w)}(b), \phi_{g, u, w}(b)\right) \arrow{dr}& \\
        (a, b, t(b)) \arrow{d} \arrow{ur} & & \left(s_{F(u, v)}(a), \phi_{g, u, v}(a), \phi_{g, u, w}\right) \arrow{d} \\
        \left(F_{u, v, w}(a, b), t(b)\right) \arrow{dr} & & \left(s_{F(u, v)}(a), F_{gu, gv, gw}(\phi_{g, u, v}(a), \phi_{g, v, w}(b))\right) \\
        & \left(s_{F(u, v)}(a), \phi_{g, u, w}(F_{u, v, w}(a, b)))\right) \arrow[equals]{ur} & &
    \end{tikzcd}
  \end{equation*}
  with the lower right equality arrow stated by \ref{musyt composition}.
\end{proof}

\begin{construction}
  \label{sz to musyt}
  Given an external action on a Burnside functor as in \cref{def sz}, we produce a Musyt version of external action as follows.
  \begin{enumerate}
  \item From \(1\)-isomorphism
    \(F(u) \xot{s_{g, u}} \psi_{g, u} \xto{t_{g, u}} F(gu)\) produce bijection
    \[\phi_{g, u} = t_{g, u} \circ s_{g, u}^{-1}.\] 
  \item There is a \(2\)-morphism
    \[\psi_{g, u, v} \colon \psi_{g, v} \circ F(u, v) \to F(gu, gv) \circ \psi_{g, u},\]
    meaning a function \(F(u, v) \times_{F(v)} \psi_{g, v} \to \psi_{g, u} \times_{F(gu)} F(gu, gv).\)
    
    \noindent Establish a bijection \(\alpha \colon F(u, v) \to F(u, v) \times_{F(v)} \psi_{g, v}\) by
    \[F(u, v) \ni a \mapsto [a, s^{-1}_{\psi_{g, v}}(t_{F(u, v)}(a))] \in F(u, v) \times_{F(v)} \psi_{g, v}\]
    and similarly \(\beta \colon F(gu, gv) \to \psi_{g, u} \times_{F(gu)} F(gu, gv)\) by
    \[F(gu, gv) \ni b \mapsto [t^{-1}_{\psi_{g, u}}(s_{F(gu, gv)}(b)), b] \in \psi_{g, u} \times_{F(gu)} F(gu, gv).\]
    Then, \(\phi_{g, u, v}\) equals \(\beta^{-1} \circ \psi_{g, u, v} \circ \alpha\).
  \end{enumerate}
\end{construction}
\begin{lemma}
  The functions \(\phi_{g, u}\) and \(\phi_{g, u, v}\) satisfy conditions of \cref{def musyt}.
\end{lemma}
\begin{proof}
  The existence of the \(2\)-morphism \(\psi_{g, h, u} \colon \psi_{gh, u} \to \psi_{g, hu} \circ \psi_{h, u}\) (as per \cref{sz 2-isos} of \cref{def sz}) implies that \(\phi_{gh, u} = \phi_{g, hu} \circ \phi_{h, u}\), thus satisfying \ref{musyt 1-composition} of \cref{def musyt}. In particular, the function \(\phi_{e, v}\) satisfies \(\phi_{e, v} \circ \phi_{e, v} = \phi_{e, v}\). As it is a bijection, this implies that \(\phi_{e, v} = \id_{F(v)}\). Thus, \ref{musyt identity} of \cref{def musyt} holds for \(\phi_{e, v}\). 

  \noindent Similarly, \ref{sz compatibility with group action} of \cref{def sz} implies \ref{musyt 2-composition} and further \ref{musyt identity} for the \(\phi_{e, u, v}\).
  
  \noindent To check \ref{musyt source target}, note that because \(\psi_{g, u, v}\) is a \(2\)-morphism, we have 
  \[s_{\psi_{g, v} \circ F(u, v)} = s_{F(gu, gv) \circ \psi_{g, u}} \circ \psi_{g, u, v}.\]
  Together with the equalities
  \[s_{F(gu, gv)} = \phi_{g, u} \circ s_{F(gu, gv) \circ \psi_{g, u}} \circ \beta, \quad s_{F(u, v)} = s_{\psi_{g, v} \circ F(u, v)} \circ \alpha,\]
  this yields
  \begin{equation*}
  \begin{aligned}
    \phi_{g, u} \circ s_{F(u, v)} & = \phi_{g, u} \circ s_{\psi_{g, v} \circ F(u, v)} \circ \alpha \\
                                  & = \phi_{g, u} \circ s_{F(gu, gv) \circ \psi_{g, u}} \circ \psi_{g, u, v} \circ \alpha \\
                                  & = (\phi_{g, u} \circ s_{F(gu, gv) \circ \psi_{g, u}} \circ \beta) \circ (\beta^{-1} \circ \psi_{g, u, v} \circ \alpha) \\
                                  & = s_{F(gu,gv)} \circ \phi_{g, u, v }.
  \end{aligned}
  \end{equation*}
  This proves the first part of \ref{musyt source target}, and the statement about target maps is shown analogously.

  \noindent Finally, \ref{musyt composition} follows from \ref{sz compatibility with composition} of \cref{def musyt} in a similar manner.
\end{proof}

\begin{prop} (Musyt \(\to\) SZ \(\to\) Musyt)
  \label{musyt sz musyt}
  If Musyt data \(\{\widetilde{\phi}_{g, v}\}_{g \in G, v \in \two{n}},\) \(\{\widetilde{\phi}_{g, u, v}\}_{g \in G, u, v \in \two{n}}\) is obtained from Musyt data \(\{{\phi}_{g, v}\}_{g \in G, v \in \two{n}},\) \(\{{\phi}_{g, u, v}\}_{g \in G, u, v \in \two{n}}\) by performing \cref{musyt to sz} and then \cref{sz to musyt}, then \(\widetilde{\phi}_{g, v} = \phi_{g, v}\) and \(\widetilde{\phi}_{g, u, v} = \phi_{g, u, v}\) for all \(g \in G,\) \(u, v \in \two{n}.\)
\end{prop}
\begin{proof}
  We write out the identities in full.
  \begin{enumerate}[(1)]
  \item \(\phi_{g, u} \mapsto (F(u) \xot{\id} F(u) \xto{\phi_{g, u}} F(gu)) \mapsto \phi_{g, u} \circ \id = \phi_{g, u}\).
  \item \(\tilde{\phi}_{g, u, v}\) maps \(a \in F(u, v)\) as follows:
    \begin{align*}
      a & \xmapsto{\alpha} [a, s^{-1}_{\psi_{g, v}}(t_{F(u, v)}(a))] = [a, t_{F(u, v)}] \\
        & \xmapsto{\psi_{g, u, v}} [\phi^{-1}_{g, u}(s_{F(gu, gv)}(\phi_{g, u, v}(a))), \phi_{g, u, v}(a)] =  [t_{\psi_{g, u}}^{-1}(s_{F(gu, gv)}(\phi_{g, u, v}(a))), \phi_{g, u, v}(a)] \\
        & \xmapsto{\beta^{-1}} \phi_{g, u, v}(a).    
      \end{align*}
  \end{enumerate}
\end{proof}

Analogous considerations exhibit the other equivalence.

\begin{prop}(SZ \(\to\) Musyt \(\to\) SZ)
  \label{sz musyt sz}
  If the Stoffregen-Zhang external action \(\tilde{\psi}_{g, v}\), \(\tilde{\psi}_{g, h, v}\), \(\tilde{\psi}_{g, u, v}\) is obtained from external action \(\psi_{g, v}\), \(\tilde{\psi}_{g, h, v}\), \(\tilde{\psi}_{g, u, v}\) by applying \cref{sz to musyt} and \cref{musyt to sz}, the two are equivariantly naturally isomorphic.
\end{prop}


\subsection{Musyt and BPS}
We reproduce first the (non-equivariant) comparison maps of \cite{Lawson_2020}, and augment them to an equivariant equivalence of equivariant cubical flow categories and Musyt's version of external actions on Burnside functors.

  Note that our definition of an equivariant cubical flow category (\cref{bps e-cfc definition}) consists of a flow category \(\mathcal C\) along with a covering \(f \colon \susp^V \mathcal C \to \cube_{\sigma}(n)\). The virtual representation \(V\) plays a role in defining the geometric realization of this flow category. However, for the duration of this section this is irrelevant, and we will assume \(V = \{0\}\), writing \(f \colon \mathcal C \to \cube_{\sigma}(n)\).

\begin{construction}(\cite[Construction 4.17]{Lawson_2020})
  \label{lls cube to burnside}
  Given a cubical flow category \(f \colon \mathcal C \to C(n)\), construct a Burnside functor \(F \colon \two{n} \to \mathcal B\) as follows:
  \begin{itemize}
  \item for \(v \in \{0, 1\}^n,\) \(F(v) = f^{-1}(v),\)
  \item for \(v > w,\) \(F(v, w)\) is the set of path components of
    \[\coprod_{x \in f^{-1}(v), \enspace y \in f^{-1}(w)} \hom(x, y),\]
    with the source map \(F(v, w) \to F(v)\) sending the components coming from \(\hom(x, y)\) to \(x \in F(v),\) target map \(F(v, w) \to F(w)\) sending those to \(y \in F(w),\)
  \item \(F(u, v, w)\) for \(u > v > w\) is induced by the continuous composition map in \(\mathcal C.\)
  \end{itemize}
\end{construction}

\begin{construction}(\cite[Construction 4.19]{Lawson_2020})
  \label{lls burnside to cube}
  Given a Burnside functor \(F \colon \two{n} \to \mathcal B,\) construct a cubical flow category \(f \colon \mathcal C \to C(n)\) as follows:
  \begin{itemize}
  \item \(\ob(\mathcal C) = \disjointunion_{v \in \{0, 1\}^n} F(v),\) the functor \(f\) sending an object \(x \in F(v)\) to \(v,\)
  \item for any \(x \in \mathcal C,\) \(\hom(x, x)\) consists only of the identity morphism,
  \item for \(x, y \in \mathcal C\) with \(v = f(x) > f(y) = w,\) consider
    \[B_{x, y} = s^{-1}(x) \cap t^{-1}(y) \subseteq F(v, w)\]
    (``the set of arrows \(x \to y\)'') and let
    \[\homs(x, y) = B_{x, y} \times \cube(n)(v, w),\]
    with map \(f \colon \homs(x, y) \to \homs(f(x), f(y))\) the projection,
  \item the composition in \(\mathcal C\) is the map
    \[F(u, v, w) \times \circ \colon (B_{y, x} \times \cube(n)(v, w)) \times (B_{x, y} \times \cube(n)(u, v)) \to (B_{x, z} \times \cube(n)(u, w));\]
    that is to say, \(F(u \to v \to w)\) is applied in the \(B_{x, y}\) factor and the product map of permutohedra \(\circ\) is applied to the \(\cube(n)(u, v)\) factor.
  \end{itemize}
\end{construction}

The following \cref{bps to musyt} and \cref{musyt to bps} should be seen as extensions of \cref{lls cube to burnside} and \cref{lls burnside to cube}, respectively.

\begin{construction}
  \label{bps to musyt}
  Apply \cref{lls cube to burnside} to obtain a Burnside functor \(F \colon \two{n} \to \burnside\) from data of a  cubical flow category \((\mathcal C, f \colon \mathcal C \to \cube_\sigma(n))\). Given a \(G\)-equivariant structure on \((\mathcal C, f \colon \mathcal C \to \cube_\sigma(n))\), consisting of functors \(\mathcal G_g \colon \mathcal C \to \mathcal C\), we construct Musyt data of an external action on \(F\). In the following we adapt naming conventions from \cref{lls cube to burnside} and \cref{bps def}.
  \begin{itemize}
  \item 
    Given \(\mathcal G_g \colon \mathcal C \to \mathcal C\) commuting with the action on \(\cube_\sigma(n)\), we let 
    \[\phi_{g, v} \coloneqq \restrict{\mathcal G_g}{f^{-1}(v)} \colon F(v) \to F(g \acts v)\]
  \item 
    Since \((\mathcal G_g)_{x, y} \coloneqq \restrict{\mathcal G_g}{\homs_{\mathcal C}(x, y)}\) are diffeomorphisms, they induce a map of sets
    \[\phi_{g, v, w} \coloneqq \pi_0 \circ \restrict{\mathcal G_g}{\coprod_{f(x) = v,  f(y) = w} \homs_{\mathcal C}(x, y)} \colon F(v, w) \to F(g \acts v, g \acts w).\]
  \end{itemize}
  Because \(\mathcal G_g\) has inverse \(\mathcal G_{g^{-1}}\), all the maps \(\phi_{g, v}\) and \(\phi_{g, v, w}\) are bijections.
  The other Musyt data axioms follow, briefly:
  \begin{itemize}
  \item \ref{musyt identity} from \ref{bps identity},
  \item \ref{musyt 1-composition} and \ref{musyt 2-composition} from \ref{bps group condition},
  \item \ref{musyt source target} trivially from the definition of target- and source maps in \ref{lls cube to burnside},
  \item \ref{musyt composition} from \ref{bps compatibility condition}.
  \end{itemize}
\end{construction}

\begin{construction}
  \label{musyt to bps}
  Given a Burnside functor with an external action, we take \cref{lls burnside to cube} as the definition of the cube flow category, and then define functors \(\mathcal G_g \colon \mathcal C \to \mathcal C\).
  We do this in such a way that conditions of \cref{bps simplification lemma} are satisfied, and so only a few conditions in \cref{bps def} need to be checked.

  First, the Musyt definition contains an acion of \(G\) on \(\two{n}.\) This extends to a \(G\)-equivariant flow category structure on \(\cube_\sigma(n),\) as per \cref{cube action gives cube flow cat}.
  
  We continue by defining \(\mathcal G_g\):
  \begin{enumerate}
  \item Consider \(\disjoint_{v \in \two{n}} \phi_{g, v} \colon \ob(\mathcal C) = \disjoint_{v \in \two{n}} F(v) \to \disjoint_{v \in \two{n}} F(gv) = \ob(\mathcal C).\) This we take to define \(\mathcal G_g\) on objects.
  \item \(\phi_{g, v, w}\) restrict to bijections \(B_{x, y} \to B_{\phi_{g, v}(x), \phi_{g, w}(y)}\), and this we extend to
    \[\hom(x, y) = B_{x, y} \times \modulispace_{C(n)}(v, w) \to B_{\phi_{g, v}(x), \phi_{g, w}(y)} \times \modulispace_{\cube_\sigma(n)}(gv, gw)\]
    by taking maps of permutohedra from the fixed \(G\)-equivariant flow category structure on the cube \(\cube_\sigma(n).\)
  \end{enumerate}
  These are functors of flow categories because \(\widetilde{\mathcal G}_g \colon C(n) \to C(n)\) are as well and because of \ref{musyt composition} in \cref{def musyt}. 
\end{construction}

\begin{prop}(Musyt \(\to\) BPS \(\to\) Musyt)
  \label{musyt bps musyt}
  Applying \cref{musyt to bps} and then \cref{bps to musyt} yields the identity on Musyt data of external action on the Burnside functor.
\end{prop}
We omit the proof which is similar to that of \cref{musyt sz musyt}.

\begin{prop} (BPS \(\to\) Musyt \(\to\) BPS)
  \label{bps musyt bps}
  Let \(\mathcal C\) be a cubical flow category. Let \(\mathcal D\) be the cubical flow category obtained from \(\mathcal C\) by applying \cref{bps to musyt} and then \cref{musyt to bps}. Then \(\mathcal C \) and \(\mathcal D\) are equivariantly naturally isomorphic.
\end{prop}
\begin{proof}
  The cubical flow categories \(\mathcal C\) and \(\mathcal D\) have the same sets of objects and actions of \(G\) on \(\ob(\mathcal C)\) and \(\ob(\mathcal D)\) agree.

  The action maps on morphism spaces take the form
  \[\left(\widetilde{\mathcal G}\right)_{x, y} = \restrict{\phi_{g, w}}{B_{x, y}} \times P_g \colon \enspace B_{x, y} \times \mathcal M_{C(n)}(v, w) \to B_{\phi_{g, v}(x), \phi_{g, w}(y)} \times \mathcal M_{\cube_\sigma(n)}(gv, gw),\]
  where
  \(P_g \colon \mathcal M_{\cube_\sigma(n)}(v, w) \to \mathcal M_{\cube_\sigma(n)}(gv, gw)\)
  is the map of morphism spaces contained in the equivariant flow category structure on the cube flow category \(C(n).\)

  The equivariant natural equivalence \(F \colon \mathcal C \to \mathcal D\) is given as follows.
  \begin{itemize}
  \item On objects, it is the identity.
  \item On morphism spaces, it maps the moduli space \(\mathcal M_{\mathcal C}(x, y)\) to \(\mathcal M_{\mathcal D}(x, y) = B_{x, y} \times \mathcal M_{\cube_\sigma(n)}(f(x), f(y))\) via \(f \colon \mathcal C \to \cube_\sigma(n).\) Formally,
    \[(F)_{x, y} = \pi_0 \times (f)_{x, y}.\]
  \end{itemize}

  This \(F\) admits a uniquely defined inverse, since \((f)_{x, y}\) are trivial covering maps.

  The functors are equivariant because \(C(n)\) has the equivariant flow category structure and the functions \(\phi_{g, v}\) satisfy the group law.
\end{proof}


\subsection{Proof of \cref{combtheorem}}
We are now ready to prove

\combtheorem
\begin{proof}
  \cref{musyt sz musyt} and \cref{sz musyt sz} show that Musyt's and Stoffregen-Zhang's notions of external action on a Burnside functor are equivalent. Likewise, \cref{musyt bps musyt} and \cref{bps musyt bps} give an equivalence of Musyt's external action on a Burnside functor with the notion of an equivariant cubical flow category \((\mathcal C, f \colon \mathcal C \to \cube_\sigma(n)).\)
  
  The equivalence is extended trivially: \((F \colon \two{n} \to \burnside, \psi)\) is the Burnside functor with external action associated to \((\mathcal C, f \colon \mathcal C \to \cube_\sigma(n))\), then to \((\mathcal C, f \colon \susp^V \mathcal C \to \cube_{\sigma}(n))\) we associate \((V, F \colon \two{n} \to \burnside, \psi)\).
  Since the shift by \(V\) plays no role in \cref{bps to musyt} and \cref{musyt to bps}, the full result follows.
\end{proof}

\begin{versiona}
In the remainder of this section, we give a more detailed description of the comparison map which establishes \cref{combtheorem}. This will be of some use in the next section.

The (Stoffregen-Zhang notion of) external action on a Burnside functor \(F \colon \two{n} \to \burnside\) consists of \(1\)-morphism \(\psi_{g, v}\) and \(2\)-morphisms \(\psi_{g, h, v}\) and \(\psi_{g, u, v}\). If they are obtained from an equivariant cubical flow category as above, they can be represented as follows:
\[\psi_{g, v} = \left[f^{-1}(v) \xot{id} f^{-1}(v) \xto{\restrict{\mathcal G_g}{f^{-1}(v)}} f^{-1}(gv),\right]\]
\[\psi_{g, h, v} = TODO\]
\[\psi_{g, u, v} = TODO\]
\end{versiona}


\section{Equivalence of realizations}
\label{sec:equivalence-result}
Building on the comparison map of the previous section, the aim of the following is to prove that the geometric realizations of an equivariant cubical flow category and that of its associated Burnside functor with external action are equivariantly stably homotopy equivalent.
\label{sec: main course}
\subsection{Homotopy coherent diagrams from neat embeddings}
\label{embeddings give hkoh}
We now show how framing a cubical flow category defines a homotopy coherent diagram over the cube, and follow up by showing how the structure of a framed equivariant cubical neat embedding yields an external action on this diagram. The results of this section can be seen as an extension of those contained in the proof of \cite[Theorem 8.]{Lawson_2020} to the equivariant setting.

Suppose we are given an equivariant framed cubical flow category \((f \colon \Sigma^V \mathcal C \to \cube_\sigma(n), \iota)\), where \(\cube_\sigma(n)\) is the topological category with group action by \(\integer_m\), induced from the \(\mathbb Z_m\)-action on \(\two{n} = \two{n'm}\). The extended equivariant cubical neat embedding  \(\tilde{\iota}\) (as in \Cref{e-fne}) furnishes a topological diagram \(\cube_\sigma(n) \to \tops_*\), meaning a homotopy coherent diagram \(\two{n} \to \tops_*\), along with an external action by \(\integer_m\). Namely, we let
\[B_x = \prod_{i = 0}^{\abs{u} - 1} B_R(V)^{d_i} \times \prod_{i = \abs{u}}^{n-1} B_{\varepsilon}(V)^{d_i}
\quad \text{and} \quad
F(u) = \coprod_{f(x) = u} B_x/{\partial \coprod_{f(x) = u} B_x}.\]

In order to define a star map associated to a morphism \(u \to v\) in \(\two{n}\), we consider the equivariant map
\[\bar{\iota}_{x, v} \coloneqq \coprod_{f(y) = v} \bar{\iota}_{x, y} \colon \coprod_{f(x) = u, f(y) = v} \left[\prod_{i=\abs{v}}^{\abs{u}-1} B_{\varepsilon}(V)^{d_i} \right] \times \modulispace_{\mathcal C}(x, y) \to E(V)_{u, v}  = \prod_{i=\abs{v}}^{\abs{u}-1} B_R(V)^{d_i} \times \cube_\sigma(n)(u, v).\]
As \(f_{x, y} \colon \modulispace_{\mathcal C}(x, y) \to \cube(n)(u, v)\) is a finite, trivial covering, there is a diffeomorphism \(\modulispace_{\mathcal C}(x, y) \cong (\cube_\sigma(n)(u, v))^{\pi_0(\modulispace_{\mathcal C}(x, y))}.\) Granted this, we can rewrite \(\bar{\iota}_{x, v}\) as
\[\bar{\iota}_{x, v} \colon \coprod_{\substack{f(y) = v \\\gamma \in \pi_0(\modulispace_{\mathcal C}(x, y))}} \left[\prod_{i=\abs{v}}^{\abs{u}-1} B_{\varepsilon}(V) ^{d_i} \right] \times \cube(n)(u, v) \to E(V)_{u, v},\]
while \ref{framed embedding 1} of \Cref{e-fne} assures that this map is the identity on the \(\cube(n)(u, v)\)-components; hence, by abuse of notation we describe it as a continuous assignment
\[\tilde{\iota}_{x, v} \colon \cube_\sigma(n)(u, v) \to \tops\left(\coprod_{\substack{f(y)= v \\ \gamma \in \pi_0(\modulispace_{\mathcal C}(x, y))}} \prod_{i = \abs{v}}^{\abs{u}-1} B_{\varepsilon}(V)^{d_i}, \enspace \prod_{i=\abs{v}}^{\abs{u}-1}B_R(V)^{d_i}\right)\]
which also respects composition in \(\cube(n)(u, v)\) (essentially due to \ref{framed embedding 3} of \Cref{e-fne}); preserving the notation, we extend by identity to
\[\bar{\iota}_{x, v} \colon \cube_\sigma(n)(u, v) \to \tops\left(\coprod_{\substack{f(y) = v \\ \gamma \in \pi_0(\modulispace_{\mathcal C}(x, y))}} B_y, \enspace B_x\right).\]

Summing over \(x\) with \(f(x) = u\), we write
\[\bar{\iota}_{u, v} = \colon \cube_\sigma(n)(u, v) \to \tops\left(\coprod_{\substack{f(x) = u, f(y) = v \\ \gamma \in \pi_0(\modulispace_{\mathcal C(x, y)})}} B_y, \enspace \coprod_{f(x) = u} B_x\right).\]
Recall the Burnside functor \(F \colon \two{n} \to \burnside\) associated to \(\mathcal C\) in \cref{lls cube to burnside}; this takes value \(F(u) = f^{-1}(u)\) on vertices and associates to an edge \(u \geq v\) in \(\two{n}\) a correspondence
\[F(u, v) = \left[F(u) \ot \pi_0 \left(\coprod_{f(x) = u, f(y) = v} \modulispace_{\mathcal C}(x, y)\right) \to F(v)\right] \colon F(u) \to F(v).\]
By \ref{framed embedding 2} of \Cref{e-fne}, each of the maps \(\bar{\iota}_{u, v}(p)\), \(p \in \cube(n)(u, v)\) represents an element of \(\stars(\{B_x\}, s_{F(u, v)})\), defining a continuous map \(\bar{\iota}_{u, v} \colon \cube(n)(u, v) \to \stars(\{B_x\}, s_{F(u, v)})\). 
Thus, there is an induced continuous family of maps of spheres (of the same dimension \(\sum d_i\)):
\[\tilde{F}(u, v) = \Phi(-, F(u, v)) \circ \bar{\iota}_{u, v} \colon \cube_\sigma(n)(u, v) \to \tops_*\left(\bigvee_{f(x) = u} S_x, \enspace \bigvee_{\substack{f(x) = u, f(y) = v \\ \gamma \in \pi_0(\modulispace_{\mathcal C}(x, y))}} S_y\right).\]
By composing with the fold map induced by canonical identification \(B_y/{\partial B_y} \cong S_y\), we obtain
\({F}(u, v) \colon \cube_\sigma(n)(u, v) \to \tops_*\left(F(u), F(v)\right)\). 
Essentially by \ref{framed embedding 3} of \Cref{e-fne}, the assignments \(F(u, v)\) respect composition in \(\cube_\sigma(n)\), and thus describe a homotopy coherent diagram \(\two{n} \to \tops_*\). 

Moreover, condition \ref{framed embedding 4} on the \(\tilde{\iota}_{x, y}\) furnishes commutative diagrams
\begin{equation*}
  \begin{tikzcd}
   \cube_\sigma(u, v) \arrow{d} \arrow{r}{\bar{\iota}_{u, v}} 
   & \tops(\coprod B_y, \disjoint B_x) \arrow{d} \\
   \cube_\sigma(g \acts u, g \acts v) \arrow{r}[swap]{\bar{\iota}_{gu, gv}}
   & \tops(\coprod B_{g \acts y}, \coprod B_{g \acts x})
  \end{tikzcd}
\end{equation*}
in which the right hand vertical arrow is induced by the external action on the Burnside functor \(F\), as associated to \(\mathcal C\) in \cref{bps to musyt}. Applying \(\Phi(-, F(u, v))\) recovers \cref{g-coherent refinement equation}, and so \(\tilde{F}\) is a \(G\)-coherent refinement of \(F\). We have thus shown the following.

\begin{prop}
  Let \((F \colon \two{n} \to \burnside, \psi)\) be the Burnside functor with external action associated to an equivariant cubical flow category \((\mathcal C, f \colon \mathcal C \to \cube_{\sigma}(n))\) by applying \cref{bps to musyt} followed by \cref{musyt to sz}. Then the homotopy coherent diagram \(\tilde{F}_V \colon \two{n} \to \tops_*\) is a \(\mathbb Z_m\)-coherent spatial refinement of \(F\).
\end{prop}


\subsection{Equivalence between BPS- and SZ-realizations}

Building on the results of \Cref{embeddings give hkoh} an equivariant analog of \cite[Theorem 8]{Lawson_2020} is proved here.

\bigtheorem
\begin{proof}

  Choose a framed equivariant cubical neat embeddding \(\bar{\iota}\) of \(\mathcal C\). Consider the homotopy coherent diagram \(\tilde{F}_V \colon \two{n} \to \tops_*\) associated to \(\bar{\iota}\) in \Cref{embeddings give hkoh}. Extend \(\tilde{F}_V\) to \sloppy \({\tilde{F}_V^+ \colon \two{n}_+ \colon \two{n}_+ \to \tops_*}\) by letting \(\tilde{F}_V^+(\ast)\) equal the basepoint. The external action of \(\mathbb Z_m\) on \(\tilde{F}_V\) induces one on \(\tilde{F}_V^+\) whereby \(\hocolim \tilde{F}_V^+\) (rather, its model defined in \Cref{eq: hocolim relations}) becomes a \(G\)-cell complex with cells \(\{C'(x)\}_{x \in F(u), u \in \two{n}}\) of the form
  \[C'(x) =
    \begin{cases}
      \modulispace_{\cube_{\sigma}(n)}(u, \vec{0}) \times [0, 2] \times B_x, & u \neq \vec{0}, \\
      \{0\} \times B_x, & u = \vec{0}.
    \end{cases}
  \]
  The non-equivariant identification is proven in \cite[Proposition 6.1]{Lawson_2020}.  Taking into account the external action, the cell \(C'(x)\) becomes a \(G_{f(x)}\)-space with action split over:
  \begin{itemize}
  \item \(\modulispace_{\cube_{\sigma}(n)}(u, \vec{0})\) as in \Cref{sec: geometric permutohedra},
  \item \(B_x = \prod_{i = 0}^{\abs{u} - 1} B_R(V)^{d_i} \times \prod_{i = \abs{u}}^{n-1} B_{\varepsilon}(V)^{d_i}\) carrying the product action induced from the \(G\)-representation \(V\),
  \item \([0, 2]\), where it is trivial.
  \end{itemize}

  Similarly, \cite[Proposition 3.18]{borodzik2021khovanov} show that \(\cfcrealization{\mathcal C}\) has a \(G\)-cell complex structure with cells
  \begin{align*}
      C(x) & = \cell(x) = \cube_{\sigma}(n)^+(u, \vec{0}) \times \prod_{i=0}^{\abs{u} - 1} B_R(V)^{e_i} \times \prod_{i=\abs{u}}^{n-1} B_\varepsilon(V)^{e_i} \\
           &  = \begin{cases}
             \modulispace_{\cube_{\sigma}(n)}(u, \vec{0}) \times [0, 1] \times B_x, & u \neq \vec{0}, \\
      \{0\} \times B_x, & u = \vec{0}.
    \end{cases}
  \end{align*}
  
  Our comparison map \(\Psi \colon \hocolim \tilde{F}_V^+ \to \cfcrealization{\mathcal C}\) is carries \(C'(x) \to C(x)\) by quotienting \([0, 2] \to [0, 2]/{[1, 2]} \cong [0, 1]\). As per the proof of \cite[Theorem 8]{Lawson_2020}, this constitutes a well-defined map of CW-complexes. Since it has degree \(\pm 1\) on each cell, the homology Whitehead theorem implies that it is a (non-equivariant) stable homotopy equivalence.
  It is also equivariant; the point is that the boxes \(B_x\) are identical in both \(C'(x)\) and \(C(x)\), as \(\tilde{F}_V\) is obtained from \(\tilde{\iota}\) in \cref{embeddings give hkoh}.  It remains to verify that for every subgroup \(H \subseteq G\), the induced map of \(H\)-fixed points \(\Psi^H \colon (\hocolim \tilde{F}_V^+)^H \to \cfcrealization{\mathcal C}^H\) is also a stable homotopy equivalence.

  To see that, note that for any \(H \subseteq G\), \(C(x)^H\) and \(C'(x)^H\) both have the form \(\modulispace_{\cube_{\sigma}}(n)^H \times [0, k] \times B_x^H\) for \(k=1, 2\) respectively. Since
  \(\modulispace_{\cube_{\sigma}}(n)^H\)
  is again a permutohedron \(\modulispace_{\cube_{\sigma}}(n')\) for some \(n' \in \naturals\) (by \cref{prop: permutohedra fixed points}; cf. \cite[Appendix B]{borodzik2021khovanov}), the cells \(C(x)^H\) and \(C'(x)^H\)  (with \(H \cap G_x \neq \emptyset\)) describe CW decompositions of \(\cfcrealization{\mathcal C}^H\), \(\hocolim (\tilde{F}_V^+)^H\), respectively. Thereby \(\Psi^H\) is a stable homotopy equivalence by the same argument as cited for \(\Psi = \Psi^{(0)}\).

\end{proof}

As a consequence of \cite[Proposition 3.27]{borodzik2021khovanov} and \cite[Lemma 5.6]{stoffregen2018localization}, the maps \(\Psi^H\) can be seen as actually realizing the map \(\Psi\) in the above proof for fixed-point cubical flow category \(\mathcal C^H\) and the homotopy coherent diagram of fixed points \(\tilde{F}_V^H\), i.e. a stable homotopy equivalence
\[\Psi \colon \cfcrealization{\mathcal C^H} \to \hocolim (\tilde{F}_V^H)^+.\]


\section{Khovanov spectra of periodic links}
\label{sec:khovanov-periodic}
We recall the constructions of Khovanov homology and Khovanov spectra, as well as their equivariant extensions due to \cite{politarczyk_equivariant_2019}, respectively \cite{stoffregen2018localization} and \cite{borodzik2021khovanov}. We present an application of \cref{bigtheorem} to this case.
\subsection{Khovanov spectra}

Given a link diagram \(D\) with \(N\) crossings (numbered \(1\) to \(N\)), the Kauffman cube of resolutions is defined as follows. For \(v = (v_1, \dotsc, v_n) \in \ob(\two{N})\), change the \(i\)-th crossing \((\tikz[baseline=-.8ex] \node[knot under cross, knot, draw, double=Gray] {};)\) to \((\tikz[baseline=-.8ex] \node[knot vert, knot, draw, double=Gray] {};)\) if \(v_i = 0\) and to \((\tikz[baseline=-.8ex] \node[knot horiz, knot, draw, double = Gray] {};)\) if \(v_i = 1.\)

Consider now the Frobenius algebra
\(\mathcal A = \integer[X]/{(X^2)}\)
with comultiplication \(\Delta \colon \mathcal A \to \mathcal A \tensor \mathcal A\) defined by
\(\Delta(1) = 1 \tensor x + x \tensor 1,\)
\(\Delta(x) = x \tensor x.\)
The Khovanov-Burnside functor \(F =F_{Kh} \colon \two{N} \to \burnside\) is defined as follows:
\begin{itemize}
\item for \(v \in \ob(\two{N})\), 
      \(F(v) = \{1, x\}^{\text{circles in the $v$-resolution of D}}\),
\item a morphism \(u \to v\) in \(\two{N}\) corresponds to a circle being split into two or two circles being merged into one; \(F(u, v)\) is the correspondence applying the comultiplication, respectively multiplication, rule of \(\mathcal A\) to the labellings,
\item for any two chains \(u \to v \to w\), \(u \to v' \to w\) with \(u \geq_2 w\), the \(2\)-morphism
      \[F_{u, v, v', w} \colon F(v, w) \circ F(u, v) \to F(v', w) \circ F(u, v')\]
      consists of bijections 
      \[A_{a, b} \coloneqq s^{-1}_{F(v, w) \circ F(u, v)}(x) \cap t_{F(v, w) \circ F(u, v)}^{-1}(z) \to s_{F(v', w) \circ F(u, v')}^{-1}(x) \cap t_{F(v', w) \circ F(u, v')}^{-1}(z) \eqqcolon A'_{a, b}\]
    for \(a \in F(u)\), \(b \in F(w)\).
    The sets \(A_{a, b}\) and \(A'_{a, b}\) both have \(1\) or \(0\) elements in all but one case. If \(\card A_{x, z} = 2\), then the resolutions along \(u \to v \to w\) split one circle (labeled \(1\) by \(a\)) into two and then merge it back to one (labeled \(x\) by \(b\)); and necessarily, the same can be said about \(u \to v' \to w\). Namely, the morphisms \(u \to w\) correspond to surgery along two edges with endpoints alternating on a single circle \(C_u\). The endpoints cut \(C_u\) into four arcs, among which we distinguish two by the following property: you walk onto them by traveling along one of the surgery edges and turning right. Th two distinghuished arcs are labeled arbitrarily by \(1\) and \(2\), and then the two relevant circles in the \(v\)- and \(v'\)-resolutions are labeled \(C_1, C_2\), respectively \(C'_1, C'_2\). The elements of \(A_{a, b}\) and \(A'_{a, b}\) can then be identified as 
    \[\alpha = (C_u \mapsto 1) \mapsto ((C_1, C_2) \mapsto (1, x)) \mapsto (C_w \mapsto x),\]    \[\beta = (C_u \mapsto 1)\mapsto ((C_1, C_2) \mapsto (x, 1)) \mapsto (C_w \mapsto x),\]
    \[\alpha' = (C_u \mapsto 1) \mapsto ((C'_1, C'_2) \mapsto (1, x)) \mapsto (C_w \mapsto x),\] 
    \[\beta' = (C_u \mapsto 1)\mapsto ((C'_1, C'_2) \mapsto (x, 1)) \mapsto (C_w \mapsto x),\]
    so that \(\restrict{F_{u, v, v', w}}{A_{a, b}}\) can be defined by \(\alpha \mapsto \alpha'\), \(\beta \mapsto \beta'\).
\end{itemize}
By \cref{burnside extension}, the remaining values of \(F\) are determined up to natural isomorphism. For proof that \ref{burnside extension commutativity} of \cref{burnside extension} holds, see \cite[Proposition 6.1]{MR3611723}.

There is a functor \(\integer \langle - \rangle \colon \burnside \to \abelian\), defined as follows: to a set \(X \in \burnside\) associate the free abelian group \(\integer \langle X \rangle\), and to a correspondence \((A, s, t) \colon X \to Y\) the map \(\integer \langle X \rangle \to \integer \langle Y \rangle:\)
\[x \mapsto \sum_{x \in X} \# \{a \in A \suchthat s(a) = x, t(a) = y\} \cdot y.\]
The classical Khovanov homology functor \(\kh \colon \left(\two{n}\right)^{op} \to \abelian\) is the composition \(\kh(D) = \integer \langle - \rangle \circ F_{Kh}\); note that the ladybug matching data encoded by 2-morphisms is forgotten in this composition.

The \emph{Khovanov chain complex} \(\ckh(D)_*\) is defined as the shift of the totalization of the functor; namely, 
\[\ckh_n = \directsum_{\abs{v} = n} \kh(v)[n_-].\]
The differential carries the component \(\kh(u)\) to \(\kh(v)\) by the map
\((-1)^{s_{u, v}} \kh(u, v)\)
if \(u \geq_1 v\), and by the zero map otherwise. The integer \(s_{u, v}\) is defined as
\(\sum_{i=1}^{k-1} u_i,\) where \(u_k\) is the single element in \(\{1, \dotsc, n\}\) with \(u_k = 0\) and \(v_k = 1\).

The complex \(\ckh(D)\) is doubly graded. In addition to the homological grading \(\abs{v} - n_-\), a summand coming from \(\kh(v)\) carries also the \emph{quantum grading}
\[n - 3n_- + \abs{v} + \card\{\text{circles labeled by $1$}\} - \card\{\text{circles labeled by x}\}.\]
Just the same, the Burnside functor \(F\) can be seen as the sum of direct summands \(F_j\) corresponding to quantum gradings.

A result of \cite{Lawson_2020} is that the stable Burnside functor \(\susp^{n_-} F\) has a well-defined realization as a spectrum \(\khovanov(D)\), whose homology is the Khovanov homology. This is the same as our \cref{def stable realization} with \(G\) the trivial group.


\subsection{Periodic links}
An \emph{\(m\)-periodic link} is one invariant under a rotation of the sphere of order \(m\), and disjoint from the axis of that rotation. We will give a digest of the constructions of equivariant Khovanov homotopy types due to \cite{stoffregen2018localization} and \cite{borodzik2021khovanov}. For a given link, there may be more than one such rotation, defining to distinct equivariant spectra; hence, the rotation is fixed at the outset. 

For the following, let \(D\) be a link diagram with \(N = nm\) crossings, invariant under a rotation of the plane \(\rho\) of order \(m\), such that \(\rho(D) = D.\) Consequently, there is an action on the cube of resolutions, which upon numbering the crossings takes the form of the natural \(\mathbb Z_m\)-action on \((\two{n})^m \cong \two{nm}\) as in \cref{section: the cube}.

\Cite[Proposition 6.2]{stoffregen2018localization} construct an external action of \(\mathbb Z_m\) on the Khovanov-Burnside functor \(F_{Kh} \colon \two{n} \to \burnside\) using \cref{external burnside simplification}. The construction is forced in almost all cases by the group action on \((\two{n})^m\) and the non-equivariant \(F\) itself. The exceptional case is that of ladybug configurations, and the well-definedness of the action follows from the fact that ladybug configurations are invariant under planar isotopy. In parallel, \cite[Proposition 4.6]{borodzik2021khovanov} use a simplification result \cite[Lemma 3.8]{borodzik2021khovanov} and construct moduli spaces inductively, with all steps but the one pertaining ladybug configurations already forced.

It is clear that the two constructions are related by the equivalences presented here in \cref{bps-musyt-sz}. The results of \Cref{sec: main course} imply the following.

\begin{theorem}
  The equivariant stable homotopy types \(\cfcrealization{\mathcal C}\) and \(\hocolim \tilde{F}_{V}^+\) associated to a periodic link by \cite{borodzik2021khovanov} and \cite{stoffregen2018localization}, respectively, are equivariantly stably homotopy equivalent. The equivalence can be realised as \(\Sigma^\infty \Psi\), where \(\Psi\) is a cellular map depending on a choice of extended equivariant cubical framed embedding \(\tilde{\iota}\) of \(\mathcal C\).
\end{theorem}

Let \(\mathbb F\) be a field. Up to chain homotopy, the Khovanov complex \(\ckh(D; \mathbb{F})\) can be equipped with an action of \(\mathbb Z_m\), whereby it can be seen as a \(\mathbb{F}[\mathbb{Z}_m]\)-module. In \cite{politarczyk_equivariant_2019}, Politarczyk defined equivariant Khovanov homology with coefficients in a \(\mathbb{F}[\mathbb{Z}_m]\)-module \(M\) by
\[\ekh^{j, q}(D; M) = \ext^j_{\mathbb{F}[\mathbb{Z}_m]}(M; \ckh^{\bullet, q}(D; \mathbb{F})).\]
In \cite[Theorem 8.3]{borodzik2021khovanov} it is proved that \(\cfcrealization{\mathcal C}\) realizes this notion of equivariant Khovanov homology via Borel cohomology.

\begin{corollary}
  Let \(D\) be an \(m\)-periodic link diagram and \(F \colon (\two{n})^m \to \burnside\) the associated Burnside functor with external group action of \(\mathbb{Z}_m\), admitting an equivariant spatial refinement with respect to representation \(V\). For any \(\mathbb{F}[\mathbb{Z}_m]\)-module \(M\), the equivariant Khovanov homology \(\ekh^{j, q}(D; M)\) is isomorphic to the reduced Borel cohomology of \(\hocolim \tilde{F}_V^+\):
  \[\ekh^{j, q}(D; M) \cong \mathrm{H}_{\mathbb{Z}_m}^*(\hocolim \tilde{F}_V+, \homs_{\mathbb{F}}(M, \mathbb{F})).\]
\end{corollary}


\printbibliography

\end{document}